\documentclass[twoside,12pt]{article}

\usepackage{amsmath}
\usepackage{amssymb}
\usepackage{pxfonts}
\usepackage{graphicx}
\usepackage{theorem}
\usepackage{pifont}
\usepackage{color}
\usepackage[all]{xy}
\usepackage{tikz}
\usepackage{tkz-euclide}
\usepackage{pgfplots}
\usepackage{accents}
\usepackage{subcaption}

\usetkzobj{all}

\theoremheaderfont{\fontfamily{pzc}\bfseries\large}
\newtheorem{theorem}{Theorem}[section]
\newtheorem{lemma}[theorem]{Lemma}
\newtheorem{proposition}[theorem]{Proposition}

\newtheorem{corollary}[theorem]{Corollary}
\newtheorem{definition}[theorem]{Definition}

{\theorembodyfont{\rmfamily}
\newenvironment{proof}{{\flushleft \emph{Proof}:}}{\hfill\ding{110}}
\newenvironment{comment}{{\flushleft \fontfamily{pzc}\bfseries\large Comment:}}{}
\newenvironment{remark}{{\flushleft \fontfamily{pzc}\bfseries\large Remark:}}{}

\newenvironment{examples}{{\flushleft {\fontfamily{pzc}\bfseries\large Examples}:}}{\hfill $\blacktriangle\blacktriangle\blacktriangle$}

\newcommand{\brk}[1]{\left(#1\right)}          
\newcommand{\Cases}[1]{\begin{cases} #1 \end{cases}}
\newcommand{\mymat}[1]{\begin{pmatrix} #1 \end{pmatrix}}
\newcommand{\Abs}[1]{\left| #1 \right|}        

\newcommand{\deriv}[2]{\frac{d#1}{d#2}}

\newcommand{\secref}[1]{Section~\ref{#1}}
\newcommand{\figref}[1]{Figure~\ref{#1}}
\newcommand{\thmref}[1]{Theorem~\ref{#1}}
\newcommand{\defref}[1]{Definition~\ref{#1}}

\newcommand{\propref}[1]{Proposition~\ref{#1}}

\newcommand{\corrref}[1]{Corollary~\ref{#1}}

\newcommand{\beq}{\begin{equation}}
\newcommand{\eeq}{\end{equation}}

\providecommand{\half}{\frac{1}{2}}
\providecommand{\R}{\mathbb{R}}

\newcommand{\Textand}{\qquad\text{ and }\qquad}

\newcommand{\g}{\mathfrak{g}}

\newcommand{\euc}{\mathfrak{e}}
\newcommand{\Vol}{d\text{Vol}}
\newcommand{\Volume}{d\text{Vol}_\g}

\newcommand{\textVol}{\text{Vol}}
\newcommand{\cof}[1]{\vartheta^#1}

\newcommand{\M}{{\mathcal{M}}}
\newcommand{\D}{{\mathcal{D}}}

\newcommand{\N}{\mathcal{N}}

\newcommand{\id}{{\text{Id}}}

\newcommand{\vp}{\varphi}

\newcommand{\dist}{\operatorname{dist}}

\newcommand{\diam}{\operatorname{Diam}}
\newcommand{\Ball}{\scrB}
\newcommand{\SO}[1]{\text{SO}(#1)}

\newcommand{\ep}{\epsilon}

\newcommand{\Llam}{\mathfrak{L}_\Lambda}
\newcommand{\LlamK}{\mathfrak{L}_{\Kmax,\Lambda}}
\newcommand{\Oof}[1]{O\brk{\frac{1}{n^{#1}}}}
\newcommand{\Wz}{Weitzenb\"ock}

\newcommand{\Kmin}{\underaccent{\bar}{\mathcal{K}}}
\newcommand{\Kmax}{\bar{\mathcal{K}}}
\newcommand{\Lmin}{{\underaccent{\bar}{L}}}
\newcommand{\Lmax}{\bar{L}}
\newcommand{\kmax}{\bar{k}}

\newcommand{\limn}{\lim_{n\to\infty}}

\newcommand{\tM}{\tilde{\M}}

\newcommand{\e}{\varepsilon}

\newcommand{\pl}{\partial}
\renewcommand{\cof}{\vartheta}
\renewcommand{\Ball}{\mathcal{B}}

\newcommand{\ga}{\gamma}
\newcommand{\dga}{\dot{\gamma}}

\newcommand{\sangle}{\sphericalangle}

\newcommand{\dis}{\operatorname{dis}}

\newcommand{\revised}[1]{{\color{red} #1}}
\renewcommand{\revised}[1]{}

\numberwithin{equation}{section}
\numberwithin{figure}{section}

\begin{document}

\title{
Riemannian surfaces with torsion as homogenization limits of locally-Euclidean surfaces with dislocation-type singularities
}
\author{Raz Kupferman and Cy Maor \\
\\
Institute of Mathematics \\
The Hebrew University \\
Jerusalem 91904, Israel}
\date{\today}
\maketitle

\begin{abstract}
We reconcile between two classical models of edge-dislocations in solids.
The first model, dating from the early 1900s models isolated edge-dislocations as line singularities in locally-Euclidean manifolds. The second model, dating from the 1950s, models continuously-distributed edge-dislocations as smooth manifolds endowed with non-symmetric affine connections (equivalently, endowed with torsion fields). In both models, the solid is modeled as a \Wz\ manifold.
We prove, using a weak notion of convergence \cite{KM15}, 
 that the second model can be obtained rigorously as a homogenization limit of the first model, as the density of singular edge-dislocation tends to infinity.   
\end{abstract}


\section{Introduction \revised{and main results}}

\Wz\ manifolds are Riemannian manifolds endowed with a flat (non-necessarily symmetric) metric connection. 
They have been used in general relativity theory in the context of teleparallelism, and in material science in the context of continuous distributions of dislocations (see e.g. Bilby et al. \cite{BBS55} \revised{and Kr\"oner \cite{Kro81}}). 

In the theory of dislocations, the material body is modeled as a smooth manifold endowed with a Riemannian metric that represents intrinsic distance between neighboring material elements. Material defects (e.g. dislocations, disclinations and point defects) are viewed as singularities in the manifold. Edge-dislocations can be modeled as curvature dipoles: a pair of cone singularities of equal magnitudes and opposite signs. 
Even if a neighborhood of the line connecting the pair of singular points (the dislocation line) is removed, thus leaving a smooth locally-flat manifold, the resulting manifold retains the same defect and cannot be isometrically embedded in the Euclidean plane
(more precisely, the defect remains in the sense that the non-trivial monodromy of the manifold does not change, see \cite{KMS14} for details).
Therefore, mathematically speaking, we can always remove a neighborhood of the dislocation line, resulting in a smooth manifold with boundary and non-trivial topology.

Real materials often contain a large number of distributed defects. In such case, one would like to smear out the singularities (or the holes) and represent the dislocations by a smooth field. Such a representation has been in use since the 1950s, with the density of dislocations  represented by the torsion field of a \Wz\ manifold. 

The smearing out of discrete entities (known as homogenization) is a central theme in the mathematics of material science. In the context of distributed dislocations, an interesting question is how a smooth Riemannian manifold endowed with a non-symmetric connection emerges as a limit of manifolds with singularities or holes, whose only intrinsic connection is the (symmetric) Riemannian connection. On  the one hand, we seek a notion of limit that involves a smooth structure, and therefore, is beyond the scope of  convergence of metric spaces (e.g., Gromov-Hausdorff convergence; see \secref{sec:metric} for details). On the other hand, standard notions of convergence for Riemannian manifolds, such as H\"older convergence,  require the sequence of converging manifolds to be of the same diffeomorphism class as the limit. This is clearly not the case for manifolds with singularities or holes converging to a smooth simply-connected manifold.

In \cite{KM15} we defined a new weak notion of convergence for \Wz\ manifolds that encompasses the presence of edge-dislocations:

\begin{definition}
\label{def:convergence}
Let $(\M_n,\g_n,\nabla_n)$, $(\M,\g,\nabla)$ be compact $d$-dimensional \Wz\ manifolds with corners.
We say that the sequence $(\M_n,\g_n,\nabla_n)$ converges to $(\M,\g,\nabla)$ 
with $p\in\revised{[}d, \infty)$, if there exists a sequence of embeddings $F_n: \M_n\to \M$ such that:
\begin{enumerate}
\item $F_n$ is asymptotically surjective:
\[
\limn \textVol_\g (\M\setminus F_n(\M_n)) = 0.
\]
\item $F_n$ are approximate isometries: the distortion vanishes asymptotically, namely,
\[
\limn\dis F_n = 0.
\]
\item $F_n$ are asymptotically rigid in the mean:
\[
\limn \int_{F_n(\M_n)} \dist{^p}(dF^{-1}_n,\SO{\g,\g_n}) \,\Volume = 0,
\]
where $\SO{\g,\g_n}$ denotes the set of metric- and orientation-preserving linear maps $T\M|_{F_n(\M_n)}\to (F_n^{-1})^*T\M_n$\revised{, and $\Volume$ denotes the volume form induced by the metric $\g$}.
\item The parallel transport converges in the mean in the following sense: every point in $\M$ has a neighborhood $U\subset\M$, with (i) a $\nabla$-parallel frame field $E$ on $U$, and (ii) a sequence of $\nabla_n$-parallel frame fields $E_n$  on $F_n^{-1}(U)$,
such that
\[
\limn \int_{U\cap F_n(\M_n)} |(F_n)_\star E_n-E|^p_\g \,\Volume = 0.
\]
\end{enumerate}
\end{definition}

In this definition, the fact that the $\M_n$'s are not diffeomorphic to the limit $\M$ (the mappings $F_n$ are only asymptotically surjective), allows for the presence of holes. Items~2 and 3 define a weak notion of convergence of Riemannian manifolds, slightly stronger than Gromov-Hausdorff (GH) convergence. Item~4 defines the convergence of the connection. This convergence is weak in the sense that it applies to the parallel transport but not to the connection as a derivation (i.e., the Christoffel symbols may not converge).

In \cite{KM15} we showed that this sense of convergence may indeed give rise to torsion as a limit of defects. We constructed a particular sequence of manifolds with edge-dislocations (and no torsion) that converges to a smooth \Wz\ manifold with non-zero torsion. 

A natural question is whether \Wz\ manifolds can be constructed generically as limits of smooth Riemannian manifolds with torsion-free connections. In the material science context, this question amounts to whether any body that falls within the 1950's model of continuously distributed dislocation is a limit of bodies with finitely many dislocations.
In this paper we show that this is the case for essentially any compact, oriented two-dimensional \Wz\ manifold.
Our main theorem is:

\begin{theorem}
\label{th:main}
Let $(\N,\g,\nabla)$ be a compact, oriented two-dimensional \Wz\ manifold  
with corners, with a Lipschitz-continuous boundary. The connection $\nabla$ is, by definition, flat and metrically-consistent with the  metric $\g$.
Then, there exists a sequence of compact locally-Euclidean Riemannian manifolds $(\M_n,\g_n)$ with a trivial holonomy, such that 
$(\M_n,\g_n, \nabla_n)$ converges to $(\N,\g,\nabla)$ in the sense of \defref{def:convergence} for every $p\in\revised{[}2,\infty)$, where $\nabla_n$ is the Levi-Civita connection of $(\M_n,\g_n)$.
\end{theorem}

The fact that the manifolds $(\M_n,\g_n)$ have trivial holonomy implies that the parallel transport of $\nabla_n$ is path-independent. Less formally, it implies that there is no curvature hidden in the ``holes" of the manifold (such as the curvature ``charge" on the tip of a cone). The sequence itself will be constructed using edge-dislocations, i.e. pairs of cone singularities of equal magnitudes and opposite signs, so the total curvature in every defect is indeed zero.

\begin{examples} 
\Wz\ manifolds that can be obtained as a limit as described in \thmref{th:main} include the following cases:
\begin{enumerate}
\item Let $D\subset \R^2$ be an open subset of the Euclidean plane, with a smooth boundary, and let $(e_1,e_2)$ be an orthonormal frame field on $\bar{D}$. By declaring this frame field parallel we obtain a flat connection $\nabla$ which is metrically-consistent with the Euclidean metric $\euc$ on $\bar{D}$, and the triplet $(\bar{D},\euc,\nabla)$ is a \Wz\ manifold satisfying the conditions of \thmref{th:main}. The example in \cite{KM15} is of this type, with $D$ a sector of an annulus endowed with the connection obtained by declaring the orthonormal frame-field $(\pl r, r^{-1}\pl \theta)$ parallel, where $r$ and $\theta$ are polar coordinates
(\figref{fig:example_sector}). 

\item Let $\mathbb{T}^2\subset \R^3$ be the two-dimensional torus embedded in $\R^3$, with the metric $\g$ induced by this embedding. From the standard representation of $\mathbb{T}$ as $\R^2/\mathbb{Z}^2$ we obtain an orthogonal frame field $(\pl x, \pl y)$. Normalizing it and declaring it parallel, we obtain a non-symmetric metrically-consistent connection $\nabla$. This is equivalent to declaring the directions of the meridians and parallels on the torus as parallel
(\figref{fig:example_torus}). 
By \thmref{th:main}, the \Wz\ manifold $(\mathbb{T},\g,\nabla)$ can be obtained as a limit of locally Euclidean manifolds with no curvature ``charges". This is somewhat surprising, as the non-zero gaussian curvature of $(\mathbb{T}^2,\g)$ seems to come out of nowhere. However, we will see that the fact that the total gaussian curvature of $(\mathbb{T}^2,\g)$ is zero plays here a crucial role. 
\end{enumerate} 
\end{examples}

\begin{figure}
\begin{center}
\begin{subfigure}[b]{0.45\textwidth}
	\includegraphics[height=1.9in]{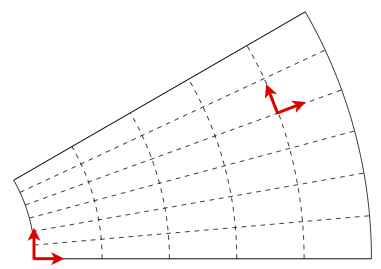}
	\caption{}
	\label{fig:example_sector}
\end{subfigure}
\quad\qquad
\begin{subfigure}[b]{0.45\textwidth}
	\includegraphics[height=1.9in]{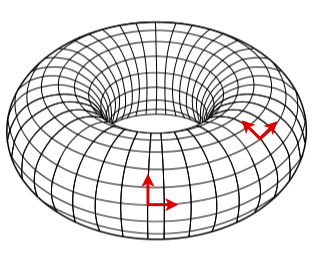}
	\caption{}
	\label{fig:example_torus}
\end{subfigure}
\end{center}
\caption{Examples of \Wz\ manifolds for which \thmref{th:main} applies. The red arrows indicate a  frame field that is parallel with respect to the metrically-consistent, non-symmetric connection on each manifold.}
\label{fig:examples}
\end{figure}

The structure of this paper is as follows: 
In \secref{sec:geo} we present properties of geodesic curves in \Wz\ manifolds. These geodesics play an important part in our construction. Since geodesics (of a general affine connection) are generally not locally length minimizing, the classical treatment to geodesics in the Riemannian setting needs to be extended.
In \secref{sec:Mn} we construct the sequence $\M_n$ of manifolds with edge-dislocations, and
in \secref{sec:mapping} we construct the embeddings $F_n:\M_n\to \N$ and establish some of their properties.
In \secref{sec:metric} we show that the distortion of $F_n$ vanishes asymptotically, which together with asymptotic surjectivity implies in particular that $\M_n$ GH-converges to $\N$.
Finally, in \secref{sec:connection} we complete the proof of \thmref{th:main} by showing the convergence of the connections.

\section{Geodesics of flat metric connections on the Euclidean plane}
\label{sec:geo}

In this section we describe properties of geodesics of flat connections. The main results  concern the existence of geodesic triangles and their properties (\corrref{cor:triangle}).

Given a general Riemannian manifold $(\N,\g)$, every metrically-consistent connection is defined by the torsion tensor $T$ (via a generalization of the Koszul formula \cite[p. 26]{Pet06}). For two-dimensional manifolds, the connection can  be defined equivalently by a vector field $V$ on $\N$ via the formula, 
\beq
\nabla_X Y = \nabla^\g_X Y + \g(X,Y) V - \g(V,Y) X,
\label{eq:V}
\eeq
where $\nabla^\g$ is the Riemannian Levi-Civita connection (see e.g. Agricola and Thier \cite{AT04}).  The torsion is then related to $V$ by
\[
T(X,Y) = \g(V,X) Y - \g(V,Y) X.
\]

The fact that a flat connection is metrically-consistent with a metric does not imply that the metric is flat (a metric is flat if the curvature tensor of the Levi-Civita connection vanishes). The following proposition relates the gaussian curvature $K$ of  $\nabla$ with the gaussian curvature $K^\g$ of $\nabla^\g$:

\begin{proposition}
\label{prop:curlfree}
Let $(\N,\g)$ be a two-dimensional Riemannian manifold. Let $\nabla$ be a metrically-consistent connection defined as in \eqref{eq:V} by a vector field $V$. Let $K$ and $K^\g$ denote the connection and the Riemannian gaussian curvatures. Then,
\[
K \,\revised{\Volume} = K^\g \,\revised{\Volume} - d \star V^\flat,
\] 
where $\star$ is the Hodge dual and $\flat$ denotes the musical isomorphism operator.
In particular, a Riemannian manifold endowed with a metrically-consistent connection $(\N,\g,\nabla)$ is a \Wz\ manifold, if and only if the curl of $V$ (viewed as a scalar) is equal to the gaussian curvature of $(\N,\g)$.
\end{proposition}

The proof is given in Appendix~\ref{sec:curlfree}. 
An immediate corollary of \propref{prop:curlfree} and the Gauss-Bonnet theorem is that a closed, oriented 2-manifold can be endowed with a metrically-consistent flat connection (i.e. a structure of a \Wz\ manifold) only if its genus is 1, that is, it is diffeomorphic to the torus. This is because the equation $d \star V^\flat = K^\g \revised{\Volume}$ cannot be solved for $V$ unless $\int_\N K^\g\, \Volume =0$. As mentioned in the introduction, the torus can indeed be endowed with a structure of a \Wz\ manifold. Of course, there  exist other compact manifolds with boundary or corners that can be endowed with such a structure.

Let $(\N,\g)$ be a Riemannian manifold, let $\nabla$ be a metrically-consistent connection, and let $\gamma$ be a curve in $\N$. $\gamma$ may be a geodesic with respect to either $\nabla$ or $\nabla^\g$ (i.e. $\nabla_{\dot{\gamma}}\dot{\gamma}=0$ or $\nabla^\g_{\dot{\gamma}}\dot{\gamma}=0$, respectively). In the former case we call $\gamma$ a \emph{geodesic}, and in the later case a \emph{segment} (with a slight abuse of terminology, we also call segment any length-minimizing curve, even if it hits the boundary). Note that while a segment is locally a length minimizer, a geodesic is not necessarily so. For an arbitrary curve $\ga$ we call  $\kappa = | \nabla^\g_{\dot{\gamma}}\dot{\gamma}|/|\dga|^2$ the \emph{curvature} of $\gamma$ (it is the geodesic curvature with respect to the Levi-Civita connection, but in order to avoid confusion we use the term ``geodesic" only with respect to the connection $\nabla$). Note that $\gamma$ has zero curvature if and only if it is a segment, and in particular a geodesic may have a non-zero curvature.

In \cite{AT04} it is shown that if $\nabla$ is defined by a vector field $V$, the curvature of a geodesic $\ga$ is given by
\beq
\kappa^2 = | V|^2 - \frac{\g(V,\dot{\gamma})^2}{|\dga|^2}.
\label{eq:kappa}
\eeq
In particular,  $\kappa\le |V|$.

Henceforth, the Riemannian manifold $(\N,\g)$ is compact, oriented, with a Lipschitz-continuous boundary, endowed with a flat metrically-consistent connection $\nabla$ defined by a smooth vector field $V$. We denote $\Lambda = \max_{\N} |V|$. 

The following proposition provides a quantitative estimate to the fact that short curves with curvature below a given bound are almost length minimizers. 

\begin{proposition}
\label{prop:geo_length}
Let $(\N,\g)$ be a two-dimensional Riemannian manifold whose gaussian curvature satisfies $K \le \Kmax$. Let $\ga:[0,\ell]\to\N$ be a curve in arclength parametrization with bounded curvature, $\kappa\le \Lambda$. Denote $d = d(\ga(0),\ga(\ell))$.
Then, there exists an $L(\Kmax,\Lambda)>0$, such that if $\ell < L(\Kmax,\Lambda)$, the following holds:
\begin{enumerate}
\item $d(\ga(0),\ga(t))$ is an increasing function of $t$.
\item
\[
d\le \ell \le \LlamK(d),
\]
where 
\[
\LlamK(x) = \Cases{
\displaystyle \frac{2}{\sqrt{\Lambda^2+\Kmax}}\sin^{-1} \brk{\sqrt{1 + \frac{\Lambda^2}{\Kmax}} \, \sin\frac{\sqrt{\Kmax} x}{2}} & \Kmax>0 \\
\\
\displaystyle \frac{2}{\Lambda}\sin^{-1} \brk{\frac{\Lambda x}{2}} & \Kmax= 0 \\
\\
\displaystyle \frac{2}{\sqrt{\Lambda^2-\Kmax}}\sin^{-1} \brk{\sqrt{1 - \frac{\Lambda^2}{\Kmax}} \, \sinh\frac{\sqrt{-\Kmax} x}{2}} & \Kmax< 0.
}
\]
In particular, for small $x$,
\[
\LlamK(x) = x + O(x^3). 
\]
\end{enumerate}
\end{proposition}

\begin{proof}
Let
\[
L_1(\Kmax,\Lambda) = \Cases{
\frac{2}{k} \tan^{-1} \frac{k}{\Lambda} & \Kmax > 0 \\
\frac{2}{\Lambda} & \Kmax = 0 \\
\frac{2}{k} \tanh^{-1} \frac{k}{\Lambda} & \Kmax < 0,
}
\]
where $k = \sqrt{|\Kmax|}$, and 
\[
L(\Kmax,\Lambda) = 
\min \brk{\half\operatorname{inj}(\N,\g), L_1(\Kmax,\Lambda)},
\]
where $\operatorname{inj}$ denotes the injectivity radius; for $\Kmax>0$, 
$\operatorname{inj}(\N,\g) \le \pi/\sqrt{\Kmax}$.
Consider a semi-geodesic (polar) parametrization $(r,\theta)$ around $\ga(0)$, in which the metric has the following form,
\[
\g(r,\theta) = \mymat{1 & 0 \\ 0 & \vp^2(r,\theta)},
\]
where $r$ is the distance from the origin $\ga(0)$, $\theta\in[0,2\pi)$, and $\vp$ is monotonically increasing in $r$, with initial conditions $\vp(0)=0$ and $\vp_r(0)=1$. 
Since for $\Kmax>0$
\beq
\label{eq:|ga|}
d(\ga(t),\ga(0))  \le L(\Kmax,\Lambda) \le \half\operatorname{inj}(\N,\g),
\eeq
it follows that $\ga$ lies within the domain of this parametrization. 

The gaussian curvature $K(r,\theta)$ is related to the function $\vp(r,\theta)$ by the well-known formula,
\[
K = -\frac{\vp_{rr}}{\vp}.
\]
Define
\[
\psi(r) = \Cases{\frac{1}{k}\,\sin(kr) & \Kmax>0 \\
r & \Kmax=0 \\
\frac{1}{k}\,\sinh(kr) & \Kmax<0.}
\]
The case $\vp=\psi$ corresponds to a surface of constant gaussian curvature $\Kmax$. 
It is easy to see that for all $(r,\theta)$ in the domain of parametrization,
\beq
\frac{\vp_r}{\vp} \ge \frac{\psi_r}{\psi}.
\label{eq:phi_psi}
\eeq

The equations of a curve $\gamma(t) = (r(t),\theta(t))$ whose (signed) curvature is $\kappa(t)$ are
\[
\begin{gathered}
\ddot{r} - \frac{\vp_r}{\vp} (\vp\dot{\theta})^2  = 
-  \kappa\, \vp\dot{\theta} \\
\vp\ddot{\theta}  + \frac{\vp_\theta}{\vp^2}(\vp\dot{\theta})^2 + 2\frac{\vp_r}{\vp}\dot{r}\, \vp\dot{\theta} = \kappa\, \dot{r}.
\end{gathered}
\]
For a curve in arclength parametrization,
\beq
\dot{r}^2 + \vp^2\dot{\theta}^2 = 1.
\label{eq:arclength_speed}
\eeq

Taking the equation for $r$, using the bound \eqref{eq:phi_psi}, the bound $|\kappa|\le\Lambda$ and \eqref{eq:arclength_speed}, we obtain the inequality,
\[
\ddot{r} - \frac{\psi_r}{\psi} (1- \dot{r}^2)  \ge  - \Lambda \sqrt{1-\dot{r}^2}.
\]
Introducing $G  = \psi \sqrt{1-\dot{r}^2}$ it follows that
\[
\dot{G} \le \Lambda \psi \dot{r}.
\]

Setting
\[
\Psi(r) = \int_0^r \psi(s)\, ds = \Cases{
\frac{1}{k^2}(1 - \cos(k r)) & \Kmax > 0 \\
r^2/2 & \Kmax = 0 \\
\frac{1}{k^2}(\cosh(k r)-1) & \Kmax < 0,
}
\]
we get upon a first integration,
\[
\psi \sqrt{1-\dot{r}^2} \le \Lambda \Psi.
\]
Thus,
\[
\dot{r}^2 \ge 1 - \Lambda^2\frac{\Psi^2}{\psi^2}.
\]
In particular, $\dot{r}$ does not change sign as long as $\Lambda\Psi < \psi$, i.e., as long as
\[
r < L_1(\Kmax,\Lambda) ,
\]
which holds since $r<L(\Kmax,\Lambda)\le L_1(\Kmax,\Lambda)$.
This proves Item~1.
Isolating $\dot{r}$ and integrating the resulting inequality one more time we obtain Item~2.

\end{proof}

\begin{remark}
\begin{enumerate}
\item The bound $L(\Kmax,\Lambda)$ is not optimal, but it is sufficient for our construction.
\item In the above proof we did not consider the case in which either $\ga$ or the segment between its endpoints intersect the boundary. While this is immaterial for the rest of the construction, the proof can be slightly modified to include these cases too.
\item The bound in Item~2 can also be obtained by proving that if $\ell<L(\Kmax,\Lambda)$ then $\ga$ is contained in a ball of radius $d/2$ around the midpoint of the segment connecting $\ga(0)$ and $\ga(\ell)$, and then using the main result in \cite{Dek80}.
\end{enumerate}
\end{remark}

\begin{corollary}
\label{cor:dist_geo_seg}
Under the assumptions of \propref{prop:geo_length} and a lower bound $K\ge \Kmin$ on the gaussian curvature, the Hausdorff distance between $\ga$ and the segment $\sigma$ that connects $\ga(0)$ with $\ga(\ell)$ is $O(d^2)=O(\ell^2)$.
\end{corollary}

\begin{proof}
Reparametrize $\ga$ and $\sigma$ such that they are defined on the interval $[0,1]$ with constant speed, i.e. $|\dga| = \ell$ and $|\dot{\sigma}| = d$.
By \propref{prop:geo_length}, $\ell = d + O(d^3)$, and for every $t\in [0,1]$,
\[
d(\ga(0),\ga(t)) = t\cdot d + O(d^3) \qquad d(\ga(t),\ga(1)) = (1-t) \cdot d + O(d^3).
\]
Consider the segment triangle with vertices $\ga(0)$, $\ga(t)$ and $\ga(1)$. By Rauch's comparison theorem \cite[p.~215]{Kli95}, the angles $\alpha(0), \alpha(t), \alpha(1)$ of this triangle are smaller than the angles $\alpha_{\Kmax}(0), \alpha_{\Kmax}(t), \alpha_{\Kmax}(1)$ of a segment triangle with same edge lengths in a space of constant curvature $\Kmax$.
Assume $\Kmax >0$ (the other cases are analogous), and denote $\kmax = \sqrt{\Kmax}$. By the law of cosines for a space of constant positive curvature \cite[p.~340]{Pet06}, we obtain
\[
\begin{split}
\cos(\alpha_{\Kmax}(0)) &= 
\frac{\cos(\kmax\, d(\ga(t),\ga(1)))-\cos(\kmax\, d(\ga(0),\ga(1)))\cos(\kmax\, d(\ga(0),\ga(t)))}{	\sin(\kmax\, d(\ga(0),\ga(1)))\sin(\kmax\, d(\ga(0),\ga(t)))} \\
&= \frac{t  + O(d^2)}{ t + O(d^2)}.
\end{split}
\]
And similarly 
\[
\cos(\alpha_{\Kmax}(1)) = \frac{1-t + O(d^2)}{1-t + O(d^2)}.
\]
Either $t$ or $1-t$ are of order $1$. Assume that $t$ is of order $1$ (the other case is analogous). Then, $\alpha(0) \le \alpha_{\Kmax}(0) = O(d)$.
Consider next the segment triangle whose vertices are $\ga(0)$, $\ga(t)$ and $\sigma(t)$. By Toponogov's comparison theorem for hinges \cite[p.~215]{Kli95}, it follows that $d(\ga(t),\sigma(t))$ is bounded from above by the length of an edge in a triangle whose two other edges are of length $d(\ga(0),\ga(t))=t\cdot d + O(d^3)$ and $d(\ga(0),\sigma(t)) = t\cdot d$ and enclose an angle $\alpha(0)= O(d)$, in a space of constant gaussian curvature $\Kmin$. It follows from the law of cosines that this distance is $O(d^2)$, with a bound independent of $t$, hence
\[
d_H(\sigma,\ga) \le \sup_{t\in[0,1]} d(\sigma(t),\ga(t)) = O(d^2).
\] 
\end{proof}

\begin{proposition}
\label{prop:geo_angle}
Let $p,q\in \N$. Assume that there exists a geodesic $\ga$ of length less than $L(\Kmax,\Lambda)$ connecting $p$ and $q$. Then the angle $\theta$ at $p$ between $\ga$ and the segment $\sigma$ connecting $p$ and $q$ satisfies the bound
\[
\theta \le \Lambda\,\LlamK(d(p,q)) + O(d(p,q)^3) = \Lambda d(p,q) + O(d(p,q)^3).
\]
\end{proposition}

\begin{proof}
It suffices to prove the proposition for the case where $\ga$ does not intersect $\sigma$. Indeed, assuming that we prove the proposition for this case, if $\ga$ intersects $\sigma$, let $q'$ be the first point of intersection. Then,
\[
\theta \le \Lambda\,\LlamK(d(p,q')) + O(d(p,q')^3) \le \Lambda\,\LlamK(d(p,q)) +O(d(p,q)^3),
\]
where we used the fact that $\LlamK$ is monotonic.

Let $\theta_p$ and $\theta_q$ denote the signed angles between $\ga$ and $\sigma$ at the points $p$ and $q$. 
Then, the Gauss-Bonnet theorem implies that
\[
\theta_q - \theta_p = \int_{\ga\cup \sigma} k(t)\, dt + \int_{A} K \Volume,
\]
where $A$ is the area enclosed by $\ga$ and $\sigma$. Therefore,
\[
|\theta_q - \theta_p| \le \Lambda \, L(\ga) + \sup |K|\cdot \textVol(A) \le  \Lambda\,\LlamK(d(p,q)) + O(d(p,q)^3),
\]
where the last inequality follows from \propref{prop:geo_length} and \corrref{cor:dist_geo_seg}.
Since $\gamma$ does not intersect $\sigma$, $\theta_p$ and $\theta_q$ have opposite signs, therefore,
\[
\theta = |\theta_p| \le |\theta_p| + |\theta_q| = |\theta_q - \theta_p| \le \Lambda\,\LlamK(d(p,q)) + O(d(p,q)^3).
\]
\end{proof}

\begin{proposition}
\label{prop:geo_exist}
Let $p,q\in \N$ satisfy
\[
d(p,q) \le \LlamK^{-1}\brk{L(\Kmax,\Lambda)}.
\]
Then there exists a unique geodesic $\ga$ connecting $p$ and $q$ whose length is less than $L(\Kmax,\Lambda)$.
\end{proposition}

\begin{proof}
Denote by $\exp^\nabla_p:\Ball(0,L(\Kmax,\Lambda))\subset T_p\N\to\N$ the exponential map with respect to $\nabla$. That  is, $\exp^\nabla_p(v) = \sigma(1)$, where $\sigma:I\to\N$  is a geodesic with $\sigma(0)=p$ and $\dot{\sigma}(0)=v$. Proposition \ref{prop:geo_length} implies that for all $t\le L(\Kmax,\Lambda)$ and unit vectors $\xi\in T_p\N$, $d(p,\exp^\nabla_p(t\xi))$ is monotonically increasing, and
\beq
\partial \exp^\nabla_p(\Ball(0,t)) \subset 
\Ball(p,t) \setminus \Ball(p,\LlamK^{-1}(t)),
\label{eq:prop3a}
\eeq
where $\Ball$ denotes a ball in $\N$ or $T_p\N$, according to the context.
A classical argument about Riemannian geodesics implies that $d\exp^\nabla_p|_0 = \text{id}$, hence $\exp^\nabla_p$ is a local diffeomorphism at the origin. It follows that there exists $\e>0$, such that the image under $\exp^\nabla_p$ of $\Ball(0,\e)$ is a simply-connected domain that contains $p$. Since
\[
\partial \exp^\nabla_p(\Ball(0,\e)) \subset 
\Ball(p,\e) \setminus \Ball(p,\LlamK^{-1}(\e)),
\]
it follows that,
\beq
\exp^\nabla_p(\Ball(0,\e)) \supset \Ball(p,\LlamK^{-1}(\e)).
\label{eq:prop3b}
\eeq
Equations \eqref{eq:prop3a} and \eqref{eq:prop3b} imply that for all $t\le L(\Kmax,\Lambda)$,
\[
\exp^\nabla_p(\Ball(0,t)) \supset \Ball(p,\LlamK^{-1}(t))
\]
(see \figref{fig:geo_exist}).
In particular, if $q\in \Ball(p,\LlamK^{-1}(L(\Kmax,\Lambda)))$, then
\[
q\in \exp^\nabla_p(\Ball(0,L(\Kmax,\Lambda))),
\]
and there exists a geodesic of length up to $L(\Kmax,\Lambda)$ from $p$ to $q$.

As for uniqueness, two geodesics that emanate from a point $p$ cannot intersect as long as they bound a simply-connected subset of $\N$, for let $q$ be their first point of intersection. By the Gauss-Bonnet theorem for \Wz\ manifolds (see Appendix~\ref{sec:GB}), the interior angles of the geodesic ``2-gon" sum up to zero, which implies that the two geodesics coincide.
\end{proof}

\begin{figure}
\begin{center}
\begin{tikzpicture}
	\tkzDefPoint(0,2){O}
	\tkzFillCircle[R, fill=gray!20, line width=0pt](O,1.4)
	\draw [->] (-2,2) -- (2,2); 
	\draw [->] (0,0) -- (0,4); 
	\tkzText(0,5){$T_p\D$} 
	\tkzText(0.75,2.35){$\Ball(0,t)$} 
	
	\tkzDefPoint(5,2){P}
	\tkzFillCircle[R, fill=gray!10, line width=0pt](P,2)
	\fill[domain=0:360,smooth, samples=100,variable=\x, color=gray!30] plot ({5+(1.5+0.2*sin(4*\x))*sin(\x)},{2+(1.5+0.2*sin(4*\x))*cos(\x)});
	\tkzFillCircle[R, fill=gray!50, line width=0pt](P,1)
	\filldraw [black] (5,2) circle (2pt);
	\tkzLabelPoint[above](P){$p$}	
	\tkzText(5,5){$\D$} 
	
	\fill [color=gray!10] (8,1) -- (8.5,1) -- (8.5,1.5) -- (8,1.5) -- cycle;
	\fill [color=gray!30] (8,2) -- (8.5,2) -- (8.5,2.5) -- (8,2.5) -- cycle;
	\fill [color=gray!50] (8,3) -- (8.5,3) -- (8.5,3.5) -- (8,3.5) -- cycle;
	\node[right] at (9,1.25) {$\Ball(p,t)$};
	\node[right] at (9,2.25) {$\exp^\nabla_p(\Ball(0,t))$};
	\node[right] at (9,3.25) {$\Ball(p,\LlamK^{-1}(t))$};

	\tkzDefPoint(0.8,5){A1}; 
	\tkzDefPoint(4.2,5){A2}; 
	\tkzDrawVector(A1,A2);
	\tkzLabelVector(A1,A2){$\exp^\nabla_p$};
\end{tikzpicture}
\end{center}
\caption{Properties of the image of the exponential map as established in \propref{prop:geo_exist}.}
\label{fig:geo_exist}
\end{figure}

The following corollary applies the results of this section to geodesic triangles:

\begin{corollary}
\label{cor:triangle}
Let $A,B,C\in\N$ be the vertices of a triangle whose edges are segments of lengths $a,b,c$  satisfying
\[
a,b,c <  \LlamK^{-1}\brk{L(\Kmax,\Lambda)}.
\]
Then, 
\begin{enumerate}
\item Every pair of vertices is connected by a unique  geodesic.
\item These geodesics do not intersect.
\item The lengths of the geodesics are bounded by $\LlamK(a)$, $\LlamK(b)$, and $\LlamK(c)$, respectively.
\item The angle $\alpha$ between the pair of geodesics emanating from $A$ and the angle $\alpha_0$ between the pair of segments emanating from $A$ satisfy 
\[
|\alpha - \alpha_0| \le 2\Lambda l + O(l^3),
\]  
where $l=\max(a,b,c)$.
A similar relation holds for the pair of angles $\beta,\beta_0$ and $\gamma,\gamma_0$.
\item $\alpha+\beta+\gamma=\pi$.
\end{enumerate}
(See \figref{fig:triangle}.)
\end{corollary}

\begin{proof}
Since $l<L(\Kmax,\Lambda)$,
Claims~1--3 are direct consequences of \propref{prop:geo_exist}.  Claim~4 is a direct consequence of \propref{prop:geo_angle}. 
Finally, Claim~5 also follows from the Gauss-Bonnet theorem (see Appendix~\ref{sec:GB}). 
\end{proof}

\begin{figure}
\begin{center}
\begin{tikzpicture}[scale=2]
	\tkzDefPoint(0,0){A};
	\tkzDefPoint(2,0){B};
	\tkzDefPoint(1,1.72){C};
	\tkzDefPoint(-1,1.72){D};
	\tkzDefPoint(1,2.5){E};
	\tkzLabelPoint[left](A){$A$};
	\tkzLabelPoint[right](B){$B$};
	\tkzLabelPoint[above](C){$C$};
	\tkzDrawPoints(A,B,C);
	\tkzDrawSegment[dashed](A,B);
	\tkzDrawSegment[dashed](B,C);
	\tkzDrawSegment[dashed](C,A);
	\tkzLabelSegment[below right](C,A){$b$};
	\tkzLabelSegment[below](A,B){$c$};
	\tkzLabelSegment[above right](B,C){$a$};
	\tkzDefPoint(0.3,0.035){A1};
	\tkzDefPoint(0.11,0.32){A2};
	\tkzDefPoint(1.7,0.035){B1};
	\tkzDefPoint(1.76,0.32){B2};
	\tkzDefPoint(0.82,1.5){C1};
	\tkzDefPoint(1.1,1.5){C2};
	\tkzMarkAngle[size=0.5, line width=1pt, color=red](A1,A,A2);
	\tkzMarkAngle[size=0.6, line width=1pt, color=red](B2,B,B1);
	\tkzMarkAngle[size=0.6, line width=1pt, color=red](C1,C,C2);
	\tkzLabelAngle[pos=0.3](A1,A,A2){$\alpha$};
	\tkzLabelAngle[pos=0.4](B1,B,B2){$\beta$};
	\tkzLabelAngle[pos=0.4](C1,C,C2){$\gamma$};
	\draw [line width=1pt] (0,0) edge[bend left=30,distance=0.2cm] (2,0);
	\draw [line width=1pt] (0,0) edge[bend left=30,distance=0.3cm] (1,1.72);
	\draw [line width=1pt] (2,0) edge[bend left=30,distance=0.2cm] (1,1.72);

\end{tikzpicture}
\end{center}
\caption{A geodesic triangle: The vertices are connected by segments (dashed lines) whose lengths are $a,b,c$. The edges are $\nabla$-geodesics and they are represented by solid lines. The angles $\alpha,\beta,\gamma$ between the edges sum up to $\pi$. For $a,b,c = O(l)$ the lengths of the edges differ from the distances between the vertices  by $O(l^3)$ and the angles between the edges deviate from the angles between the corresponding segments by $O(l)$.}
\label{fig:triangle}
\end{figure}
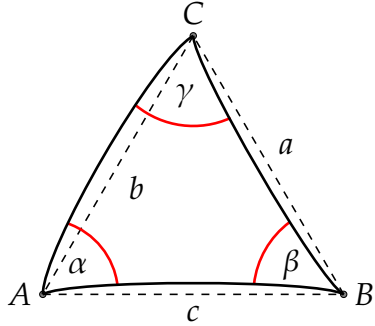

\section{Construction of locally-Euclidean manifolds, $\M_n$}
\label{sec:Mn}

In this section we construct the sequence $\M_n$ approximating $\N$ by triangulating $\N$ with geodesic triangles and replacing each triangle with a locally-Euclidean triangle that encloses an edge-dislocation. 
The following two propositions assert the existence of a regular geodesic triangulation (\propref{prop:triangulation}), and the existence of a locally-Euclidean triangle enclosing an edge-dislocation with a given boundary (\propref{prop:edge_dis}). In the remaining of the section we use these results to define $\M_n$.

From now on we will assume that the manifold $\N$ is simply-connected. We will remove this restriction in \secref{sec:connection}.

\begin{proposition}[Existence of a regular geodesic triangulation]
\label{prop:triangulation}
For every $n\in,\mathbb{N}$ large enough, there exists a subdomain $\N_n\subset \N$, such that
\beq
\label{eq:N_minus_N_n}
\N\setminus \N_n \subset \Ball(\partial\N,3/n),
\eeq
and $\N_n$ can be triangulated by geodesic triangles whose edge lengths are bounded by
\[
\frac{\Lmin}{n} \le \text{edge length} \le \frac{\Lmax}{n},
\]
for some constants $c,C>0$ independent of $n$.
The angles between intersecting edges are in bounded uniformly from $0$ and $\pi$, i.e.
\[
\delta \le \text{angle size} \le \pi-\delta
\]
for some constant $\delta\in (0,\pi)$ independent of $n$.
Moreover, the angles in each triangle sum up to $\pi$.
\end{proposition}

\begin{proof}
Let $n > 1/\LlamK^{-1}(L(\Kmax,\Lambda))$.
First, triangulate a subdomain $\N'_n\subset \N$ by segments of edge length
\[
\frac{\Lmin}{n} \le \text{edge length} \le \frac{\Lmax}{n},
\]
and angles
\[
\delta \le \text{angle size} \le \pi-\delta,
\]
such that
\beq
\label{eq:plN_n}
\partial \N_n' \subset \Ball(\partial\N,2/n) \setminus \Ball(\partial\N,1/n).
\eeq
Such a construction is always possible for large enough $n$, see e.g. \cite[~p. 157]{Ber02}.
Then take $\N_n$ to be the geodesic triangulation based on the graph structure of $\N_n'$ (see \figref{fig:triangulation}). For large enough $n$, the geodesics never hit the boundary as a result of \eqref{eq:plN_n} and \corrref{cor:dist_geo_seg}, which also imply \eqref{eq:N_minus_N_n}. The other properties of $\N_n$ for $n$ large enough are direct consequences of \corrref{cor:triangle}, up, possibly, to a slight adjustment of the constants $\Lmin$, $\Lmax$ and $\delta$.
\end{proof}

\begin{figure}
\begin{center}
\begin{tikzpicture}[scale=1.5]
	\foreach \x in {1,...,4}
	{
		\foreach \y in {1,...,2}
		{
			\tkzDefPoint(\x,2*0.86603*\y){A};
			\tkzDrawPoint(A);
			\tkzDefPoint(\x+0.5,2*0.86603*\y + 0.86603){B};
			\tkzDrawPoint(B);
		}
	}  
	\foreach \x in {1,...,4}
	{
		\foreach \y in {1,...,2}
		{
			\tkzDefPoint(\x,2*0.86603*\y){A1};
			\tkzDefPoint(\x+1,2*0.86603*\y){A2};
			\tkzDefPoint(\x+0.5,2*0.86603*\y+0.86603){A3};
			\tkzDefPoint(\x-0.5,2*0.86603*\y+0.86603){A4};
			\tkzDrawSegment[dashed](A1,A2);
			\tkzDrawSegment[dashed](A1,A3);
			\tkzDrawSegment[dashed](A1,A4);
			\draw [line width=1pt] (\x,2*0.86603*\y) edge[bend left=30,distance=0.2cm] (\x+1,2*0.86603*\y);
			\draw [line width=1pt] (\x,2*0.86603*\y) edge[bend left=30,distance=0.2cm] (\x+0.5,2*0.86603*\y+0.86603);
			\draw [line width=1pt] (\x,2*0.86603*\y) edge[bend left=30,distance=0.2cm] (\x-0.5,2*0.86603*\y+0.86603);
			
			\tkzDefPoint(\x+0.5,2*0.86603*\y+0.86603){B1};
			\tkzDefPoint(\x-0.5,2*0.86603*\y+0.86603){B2};
			\tkzDefPoint(\x+1,2*0.86603*\y+0.86603+0.86603){B3};
			\tkzDefPoint(\x,2*0.86603*\y+0.86603+0.86603){B4};
			\tkzDefPoint(\x+1,2*0.86603*\y){B5};
			\tkzDrawSegment[dashed](B1,B3);
			\tkzDrawSegment[dashed](B1,B4);
			\tkzDrawSegment[dashed](B1,B2);
			\tkzDrawSegment[dashed](B1,B5);
			\draw [line width=1pt] (\x+0.5,2*0.86603*\y+0.86603) edge[bend left=30,distance=0.2cm] (\x-0.5,2*0.86603*\y+0.86603);
			\draw [line width=1pt] (\x+0.5,2*0.86603*\y+0.86603) edge[bend right=30,distance=0.2cm] (\x+1,2*0.86603*\y);
			\draw [line width=1pt] (\x+0.5,2*0.86603*\y+0.86603) edge[bend left=30,distance=0.2cm] (\x+1,2*0.86603*\y+0.86603+0.86603);
			\draw [line width=1pt] (\x+0.5,2*0.86603*\y+0.86603) edge[bend left=30,distance=0.2cm] (\x,2*0.86603*\y+0.86603+0.86603);
		}
	}  
\end{tikzpicture}
\end{center}
\caption{Triangulation of the domain $\N_n$ by geodesic triangles. The dashed lines are segments of length $O(1/n)$ triangulating $\N'_n$. The solid lines are the geodesic curves connecting those vertices.}
\label{fig:triangulation}
\end{figure}
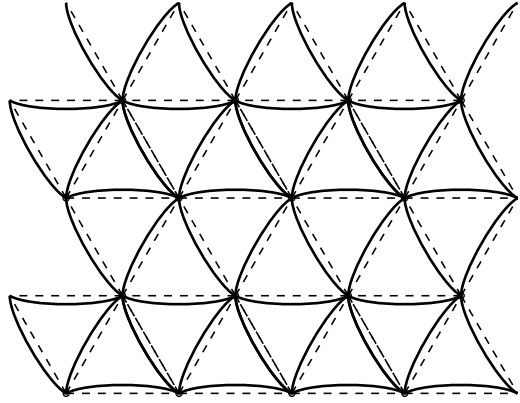

\begin{corollary}
\label{cor:vertex_edge_dist}
In the triangulation described in \propref{prop:triangulation}, the distance between a vertex and the opposite edge in a triangle is larger than $r/n$ for some $r>0$ independent of $n$.
\end{corollary}

\begin{proof}
By \corrref{cor:dist_geo_seg}, it is sufficient to consider the segment triangulation used in the proof of \propref{prop:triangulation}, instead of the geodesic triangulation, since the distance between the opposite segment and the opposite edge is $O(n^{-2})$.

Assume by contradiction that there exists a sequence of triangles, with the $n$-th triangle belonging to the $n$-th triangulation, such that the distance between one of the vertices and its opposite edge is $o(1/n)$.
Consider the segment triangle $ACE$ in \figref{fig:vertex_edge_dist}. By Rauch's comparison theorem for triangles (see \cite[p.~215]{Kli95}), the angles $\alpha, \gamma_1 , \epsilon_1$ are smaller than the angles $\alpha', \gamma_1', \epsilon_1'$ of the segment triangle with same side lengths in a space of constant curvature $\Kmax$ (the upper bound on the gaussian curvature). 
Assume $\Kmax>0$ (the other cases are analogous). In a space of constant curvature the law of sines reads
\[
\frac{\sin\brk{\sqrt{\Kmax} d}}{\sin\brk{\sqrt{\Kmax} b}} = \frac{\sin(\alpha')}{\sin(\epsilon_1')} \ge \sin(\alpha')
\]
Since $b\in (\Lmin /n,\Lmax /n)$, it follows that if $d=o(1/n)$, then $\sin(\alpha') = o(1)$. Since $\alpha'\ge\alpha>\delta$, it follows that $\pi -\alpha' = o(1)$, hence $\gamma_1' + \epsilon_1' = o(1)$ ($\alpha'  + \gamma_1'+\epsilon_1' = \pi + O(1/n^2)$ by the Gauss-Bonnet theorem). This is a contradiction, since 
\[
\gamma_1' + \epsilon_1' \ge \gamma_1 + \epsilon_1 = \pi - \alpha + \Oof{2} > \delta + \Oof{2}.
\]
\end{proof}
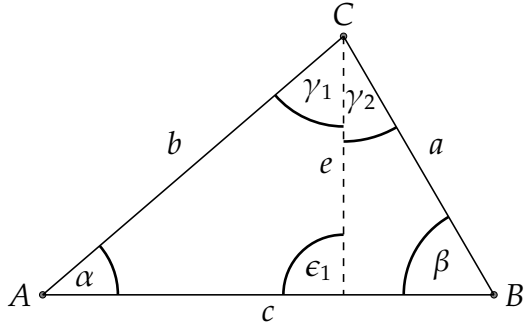
\begin{figure}
\begin{center}
\begin{tikzpicture}[scale=2]
	\tkzDefPoint(-1,0){A};
	\tkzDefPoint(2,0){B};
	\tkzDefPoint(1,1.72){C};
	\tkzDefPoint(1,0){E};
	\tkzLabelPoint[left](A){$A$};
	\tkzLabelPoint[right](B){$B$};
	\tkzLabelPoint[above](C){$C$};
	\tkzDrawPoints(A,B,C);
	\tkzDrawSegment(A,B);
	\tkzDrawSegment(B,C);
	\tkzDrawSegment(C,A);
	\tkzDrawSegment[dashed](C,E);
	\tkzLabelSegment[above left](C,A){$b$};
	\tkzLabelSegment[below](A,B){$c$};
	\tkzLabelSegment[above right](B,C){$a$};
	\tkzLabelSegment[left](C,E){$e$};
	\tkzMarkAngle[size=0.5, line width=1pt](B,A,C);
	\tkzMarkAngle[size=0.6, line width=1pt](C,B,A);
	\tkzMarkAngle[size=0.6, line width=1pt](A,C,E);
	\tkzMarkAngle[size=0.7, line width=1pt](E,C,B);
	\tkzMarkAngle[size=0.4, line width=1pt](C,E,A);
	\tkzLabelAngle[pos=0.3](C,A,B){$\alpha$};
	\tkzLabelAngle[pos=0.4](C,B,A){$\beta$};
	\tkzLabelAngle[pos=0.4](A,C,E){$\gamma_1$};
	\tkzLabelAngle[pos=0.45](E,C,B){$\gamma_2$};
	\tkzLabelAngle[pos=0.23](C,E,A){$\epsilon_1$};
\end{tikzpicture}
\end{center}
\caption{A segment triangle $ABC$. 
The edge lengths satisfy 
$a,b,c \in (\protect\Lmin/n,\Lmax /n)$, 
and the angles satisfy $\alpha,\beta,\gamma_1+\gamma_2 \in (\delta,\pi-\delta)$. 
$e$ is a segment that minimizes the distance between the vertex $C$ and the edge $c$.}
\label{fig:vertex_edge_dist}
\end{figure}

The next proposition shows that we can associate
with any geodesic triangle a locally-flat triangle with an edge dislocation (a pair of cone singularities of equal magnitudes and opposite signs), that has the same edge lengths and angles. These triangles will be the building block of the approximating sequence of manifolds $\M_n$.

\begin{proposition}
\label{prop:edge_dis}
Let $\theta\in (0,\pi/2)$. Let $a,b,c$ be numbers satisfying
\beq
\frac{\Lmin}{n} \le  a,b,c \le \frac{\Lmax}{n},
\label{eq:edge_cond}
\eeq
such that the sum of any two of them is larger than the third. Let $\alpha_0,\beta_0,\gamma_0$ be the angles of a (Euclidean) triangle whose side lengths are $a,b,c$. Suppose that
$\alpha,\beta,\gamma$ satisfy
\[
\alpha + \beta + \gamma = \pi.
\]
and
\beq
|\alpha-\alpha_0|, |\beta-\beta_0|, |\gamma-\gamma_0| \le \frac{K}{n}.\label{eq:angle_cond}
\eeq
Then, for $n$ large enough, there exists a locally-Euclidean triangle of edge-lengths $a,b,c$ and respective angles $\alpha,\beta,\gamma$, enclosing an edge-dislocation of magnitude $\e=O(1/n^2)$ and a disclination angle $\theta$. 
\end{proposition}

\begin{proof}
If $a,b,c$ are $\alpha,\beta,\gamma$ are compatible with Euclidean geometry then the claim is trivial with $\e=0$ (no dipole).  Otherwise, the law of sines  
\[
\frac{a}{\sin\alpha} = \frac{b}{\sin\beta} = \frac{c}{\sin\gamma}
\]
does not hold. 

Consider  \figref{fig:dipole}. It shows two (Euclidean) polygons $ADFEGC$ and $BD'F'E'G'$. The dipole is constructed by identifying the edges $DF$ with $D'F'$, $EF$ with $E'F'$ and $EG$ with $E'G'$, each pair assumed of same length. Also, 
\[
\sangle EFD = \sangle E'F'D' = \sangle FEG = \sangle F'E'G' = \pi-\theta.
\] 
The dipole magnitude is 
\[
|\e| = 2\,|EF|\,\sin\theta.
\] 
If $\e>0$ we  elongate the edge opposite to $\gamma$ and shorten the edge opposite to $\alpha$; if $\e<0$ it is the other way around.
We need to prove that such a construction is possible, under the constraints
\[
AD  + D'B = c 
\Textand
CG  + G'B = a,
\]
$\vp\in(0,\gamma)$, and $\e = O(1/n^2)$.

By means of \figref{fig:dipole} and straightforward trigonometry we obtain that $\vp$ and $\e$ are given by
\beq
\tan\vp = \frac{b -a\,\cos\gamma - c\,\cos\alpha}{a\,\sin\gamma  - c\,\sin\alpha},
\label{eq:tanvp}
\eeq
and
\[
\e  = \frac{c\,\sin\alpha - a\,\sin\gamma}{\cos\vp}.
\]

Consider \eqref{eq:tanvp}. In a Euclidean triangle both the numerator and the denominator vanish. Order the vertices of the triangle such that $a/\sin\alpha$, $b/\sin\beta$ and $c/\sin\gamma$ form a monotone sequence. The sign of the numerator of \eqref{eq:tanvp} does not change whether we take an increasing or decreasing sequence. If it is positive we let $a/\sin\alpha > c/\sin\gamma$, otherwise we let $a/\sin\alpha < c/\sin\gamma$. Thus we ensure that $\tan\vp>0$.

We now prove that $\vp<\gamma$. 
Take for example the case where
\[
\frac{a}{\sin\alpha} \ge \frac{b}{\sin\beta} \ge \frac{c}{\sin\gamma},
\]
where at least one of the inequalities is strong (the other case is similar). Then,
\[
\begin{split}
a - b\,\cos\gamma - c\,\cos\beta &> a - \frac{a\sin\beta}{\sin\alpha} \,\cos\gamma -
\frac{a\,\sin\gamma}{\sin\alpha}\,\cos\beta \\
&= \frac{a}{\sin\alpha}\brk{\sin\alpha - \sin(\beta + \gamma)} = 0.
\end{split}
\]
Therefore,
\[
\begin{split}
\frac{\tan\vp}{\tan\gamma} &= \frac{b\cos\gamma -a\,\cos^2\gamma - c\,\cos\alpha\cos\gamma}{a\,\sin^2\gamma  - c\,\sin\alpha\sin\gamma} \\
&= \frac{b\cos\gamma -a +a\,\sin^2\gamma + c \cos\beta-  c\,\sin\alpha\sin\gamma}{a\,\sin^2\gamma  - c\,\sin\alpha\sin\gamma} \\
&= 1 - \frac{1}{\sin\gamma} \frac{a-b\cos\gamma - c \cos\beta}{a\,\sin\gamma  - c\,\sin\alpha} < 1,
\end{split}
\]
which proves that indeed $\vp\in(0,\gamma)$.

It remains to show that $\e = O(1/n^2)$. From our bound on $\vp$ it follows that $\cos\vp$ is bounded away from zero  for large $n$. From Assumptions \eqref{eq:edge_cond} and \eqref{eq:angle_cond} it follows that $c\,\sin\alpha - a\,\sin\gamma= O(1/n^2)$, which implies $\e = O(1/n^2)$.
\end{proof}

\begin{figure}
\begin{center}
\begin{tikzpicture}[scale=1.5]
	\tkzDefPoint(0,0){B};
	\tkzDefPoint(3,2){Cp};
	\tkzDefPoint(2.0,2){Ep};
	\tkzDefPoint(1.5,2.5){Fp};
	\tkzDefPoint(0.7,2.5){Dp};
	\tkzDrawPolygon[fill=gray!5](B,Cp,Ep,Fp,Dp);
	\tkzDrawPoints(B,Cp,Ep,Fp,Dp);
	\tkzDefPoint(0.7,2.8){D};
	\tkzDefPoint(1.5,2.8){F};
	\tkzDefPoint(2.5,2.8){FF};
	\tkzDefPoint(2,3.3){E};
	\tkzDefPoint(3,3.3){C};
	\tkzDefPoint(3+1,3.3+0.666){newC};
	\tkzDefPoint(2.7,3.3+0.666){tmp};
	\tkzDefPoint(1.19+0.2,4.25+0.580){A};
	\tkzDrawPoints(A,D,E,F,C,newC);
	\tkzDrawPolygon[fill=gray!5](A,newC,C,E,F,D);
	\tkzLabelSegment[above](A,newC){$b$}
	\tkzDrawSegment[dashed](newC,tmp)
	\tkzLabelPoint[below](B){$B$}
	\tkzLabelPoint[above](A){$A$}
	\tkzLabelPoint[left](D){$D$}
	\tkzLabelPoint[left](Dp){$D'$}
	\tkzLabelPoint[right](C){$G$}
	\tkzLabelPoint[right](newC){$C$}
	\tkzLabelPoint[right](Cp){$G'$}
	\tkzLabelPoint[above left=0.01](F){$F$}
	\tkzLabelPoint[below left=0.01](Fp){$F'$}
	\tkzLabelPoint[left](E){$E$}
	\tkzLabelPoint[left](Ep){$E'$}
	\tkzLabelAngle[pos=0.3](D,A,newC){$\alpha$}
	\tkzMarkAngle[size=0.5](D,A,newC)
	\tkzLabelAngle[pos=0.4](Cp,B,Dp){$\beta$}
	\tkzMarkAngle[size=0.6](Cp,B,Dp)
	\tkzText(3.6,3.82){ $\gamma$}
	\tkzMarkAngle[size=0.6](A,newC,C)
	\tkzLabelAngle[pos=0.8](A,newC,tmp){$\vp$}
	\tkzMarkAngle[size=1](A,newC,tmp)
	\tkzText(2.4,1.87){{\small $\gamma-\vp$}}
	\tkzMarkAngle[arc=lll,size=0.95](Ep,Cp,B)
	\tkzDrawSegment[dashed](F,FF);
	\tkzMarkAngle[size=0.5](FF,F,E)
	\tkzLabelAngle[pos=0.35](FF,F,E){{\small $\theta$}}
\end{tikzpicture}
\end{center}
\caption{Construction of a Euclidean triangle with an edge-dislocation used in the proof of \propref{prop:edge_dis}.}
\label{fig:dipole}
\end{figure}
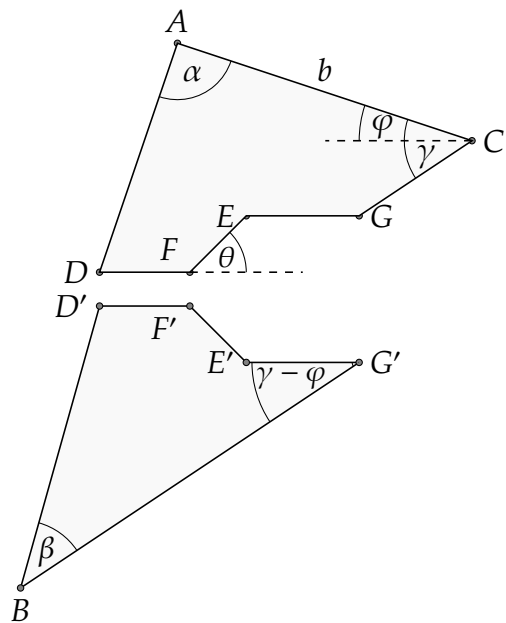

\begin{comment}
The line $EF = E'F'$, whose length is $O(1/n^2)$,  is called the \emph{dislocation line}. 
By shifting it, either horizontally or vertically, we can always modify the construction so that  the singular points $E$ and $F$, are  located at the center of the defective triangle, in the sense that their distance from the boundary is \revised{of order $1/n$}. In particular, this guarantees that when $\e<0$ the  points $E$ and $F$ remain in the interior of the triangle.
\end{comment}

\begin{definition}
Let $a,b,c$, $\alpha,\beta,\gamma$, and $\theta$ satisfy the conditions of \propref{prop:edge_dis}. We denote by
\[
\Delta(a,b,c;\alpha,\beta,\gamma;\theta)
\]
the locally-Euclidean triangle enclosing an edge-dislocation constructed in \propref{prop:edge_dis}. We assume that the dislocation line is at the center of the triangle as explained in the above comment.
\end{definition}

We now turn to construct a sequence of locally-flat manifolds with path-independent Levi-Civita connections, $(\M_n,\g_n,\nabla_n)$, that approximate the \Wz\ manifold $(\N,\g,\nabla)$. By \propref{prop:triangulation}, there exists a subdomain $\N_n\subset\N$ that can be triangulated by geodesic triangles $\N_{n,i}$ for $i\in I_n$, 
\[
\N_n = \bigcup_{i\in I_n} \N_{n,i}\,\,,
\]
whose edges $a_{n,i}, b_{n,i}, c_{n,i}$ and angles $\alpha_{n,i}, \beta_{i,n}, \gamma_{n,i}$ satisfy conditions \eqref{eq:edge_cond} and \eqref{eq:angle_cond}. Note that $|I_n| = O(n^2)$.

For every $i\in I_n$ we construct a locally-Euclidean triangle with edge-dislocation
\[
\tM_{n,i} = \Delta(a_{n,i}, b_{n,i}, c_{n,i};\alpha_{n,i}, \beta_{i,n}, \gamma_{n,i};\theta).
\]
We denote by $\tM_n$ the amalgam of the $\tM_{n,i}$, obtained by replacing each $\N_{n,i}$ by the corresponding $\tM_{n,i}$.

$\tM_n$ is a compact, simply-connected topological manifold with corners. It is smooth and locally-Euclidean everywhere except at two singular points within each triangle $\tM_{n,i}$. Removing the dislocation line that connects each pair of singular points, the Levi-Civita parallel transport $\nabla_n$ is path-independent (see \cite{KM15} for more detail about a similar construction). We denote by $\g_n$ the Riemannian metric on the smooth part of $\tM_n$ and by $\tilde{d}_n$ the induced distance function on $\tM_n$. The Riemannian metric $\g_n$ induces also a volume measure on $\tM_n$ (including the singular points).

Let $L_{n,i}\subset\tM_{n,i}$ be open neighborhoods of radius $1/n^2$ around the dislocation line, and let
\[
L_n = \bigcup_{i\in I_n} L_{n,i}.
\]

Set $\M_{n,i} = \tM_{n,i}\setminus L_{n,i}$ and define
\[
\M_n =  \bigcup_{i\in I_n} \M_{n,i} = \tM_n\setminus L_n.
\] 
$(\M_n,\g_n)$ is a smooth, compact, multiply-connected, locally-Euclidean manifold with corners. 
We denote by $d_n$ the length distance induced by $\g_n$. Since the diameter of $L_{n,i}$ is $O(1/n^2)$, it follows that the Gromov-Hausdorff distance between $(\tM_{n,i},\tilde{d}_n)$ and $(\M_{n,i},d_n)$ is $O(1/n^2)$, namely,
\[
\sup_{p,q\in \M_{n,i}} \left|\tilde{d}_n(p,q) - d_n(p,q)\right| = O\brk{\frac{1}{n^2}}.
\] 
\revised{
\begin{remark}
\begin{enumerate}
\item Note that the distances inside each triangle $\M_{n,i}$ differ only by $O(1/n^2)$ from the distances of a Euclidean triangle with the same edges. 
For $n$ large enough, $\M_{n,i}$ is convex in $\M_n$, in the sense that segments between points in $\M_{n,i}$ stays inside it (that is, the induced metric on $\M_{n,i}$ from $d_n$ and its intrinsic distance are the same). 
This is because the only way that a curve between points on an edge of $\M_{n,i}$ could have been shorter than the length of the edge is by moving around a dislocation. However, since the dislocations are located at a distance of order $1/n$ from the edge and the length gain is only order $1/n^2$, such a curve would in fact be longer.
\item Since the Levi-Civita parallel transport in each $\M_{n,i}$ is path independent (this is one of the main features of the dislocation, see \cite{KM15} for details) and since the angle around each vertex of each of the triangles $\M_{n,i}$ is exactly $2\pi$, the Levi-Civita parallel transport $\nabla_n$ in the whole manifold $\M_n$ is path-independent. 
\end{enumerate}
\end{remark}}

\section{Embeddings of $\M_n$ into the \Wz\ manifold $\N$}
\label{sec:mapping}

In this section we construct embeddings $F_n:\M_n\to \N$ satisfying the conditions of \defref{def:convergence}.
We denote by $X_n$ the skeleton formed by the union of the boundaries of the triangles with defects $\tM_{n,i}$. 
Likewise, we denote by $Y_n$ the  skeleton formed by the union of the boundaries of the geodesic triangles $\N_{n,i}$.

These skeletons have the following properties:
\begin{enumerate}
\item 
The vertices of $X_n$ form a finite $O(n^{-1})$-net of $\M_n$ and the
vertices of $Y_n$ form a finite $O(n^{-1})$-net of $\N$ of the same cardinality.
This follows from the construction and in particular from the fact that $\textVol(\N\setminus\N_n)\to0$.

\item The edges in $X_n$ are curves that are of the same length as the corresponding edges in $Y_n$.
\item $Y_n$ consists of $\nabla$-geodesics and $X_n$ consists of $\nabla_n$-geodesics.  

\end{enumerate}

It follows that there exists a natural mapping $T_n: X_n\to Y_n$ that preserves the intrinsic distance of $X_n$ and $Y_n$ (the intrinsic distances on path-connected subsets differ from the induced distances $d$ and $d_n$).

Before extending $T_n$ to smooth embeddings $\M_n\to\N$, we need the following technical lemma:

\begin{lemma}
Consider the two geometric figures displayed in \figref{fig:tedious}.
Both figures exhibit a geodesic curve of length $a$, with at one end a geodesic curve of length $b$ emanating at an angle $\gamma$ and at its other end  geodesic curve of length $c$ emanating at an angle $\beta$. The figure on the left is in the Euclidean plane (i.e., the segments $pq$, $pp_1$ and $qq_1$ are Euclidean segments), whereas the figure on the right is in $\N$ (i.e., the curves $PQ$, $PP_1$ and $QQ_1$ are $\nabla$-geodesics). It is given that
\[
a,b,c \le \frac{\Lmax}{n}.
\]
Then, there exists a constant $\Delta$ which only depends on $\Kmin$, $\Kmax$, $\Lambda$ and $\Lmax$, such that for large enough $n$,
\[
|d(P_1,Q_1) - |p_1 - q_1|| \le \frac{\Delta}{n^2}.
\]

\begin{figure}
\begin{center}
\begin{tikzpicture}[scale=2.0]
	\tkzDefPoint(0,0){p}
	\tkzDefPoint(2,0){q}
	\tkzDefPoint(0.8,0.8){p1}
	\tkzDefPoint(1.25,0.6){q1}
	\tkzDrawPoints(p,q,p1,q1)
	\tkzDrawSegment(p,q);
	\tkzDrawSegment(p,p1);
	\tkzDrawSegment(q,q1);
	\tkzLabelSegment[below](p,q){$a$};
	\tkzLabelSegment[above left](p,p1){$b$};
	\tkzLabelSegment[above right](q,q1){$c$};
	\tkzLabelPoint[ left](p){$p$}
	\tkzLabelPoint[ right](q){$q$}
	\tkzLabelPoint[above](p1){$p_1$}
	\tkzLabelPoint[above](q1){$q_1$}
	\tkzMarkAngle[size=0.5](q,p,p1);
	\tkzMarkAngle[size=0.5](q1,q,p);
	\tkzLabelAngle[pos=0.33](q,p,p1){$\gamma$};
	\tkzLabelAngle[pos=0.37](p,q,q1){$\beta$};

	\tkzDefPoint(3,0){P}
	\tkzDefPoint(5,0){Q}
	\tkzDefPoint(3.8,0.8){P1}
	\tkzDefPoint(4.25,0.6){Q1}
	\tkzDrawPoints(P,Q,P1,Q1)
	\draw  (3,0) edge[bend right=30,distance=0.2cm] (5,0);
	\draw  (3,0) edge[bend right=30,distance=0.2cm] (3.8,0.8);
	\draw (5,0) edge[bend right=30,distance=0.2cm] (4.25,0.6);
	\tkzLabelPoint[left](P){$P$}
	\tkzLabelPoint[right](Q){$Q$}
	\tkzLabelPoint[above left](P1){$P_1$}
	\tkzLabelPoint[above right](Q1){$Q_1$}
	\tkzLabelSegment[below=4pt](P,Q){$a$};
	\tkzLabelSegment[above left=-1pt](P,P1){$b$};
	\tkzLabelSegment[above right=3pt](Q,Q1){$c$};
	\tkzLabelAngle[pos=0.3](P,Q,Q1){$\beta$};
	\tkzText(3.37,0.1){$\gamma$};
	\draw [domain=-7:36] plot ({3+0.5*cos(\x)}, {0.5*sin(\x)});
	\draw [domain=-7.5:48] plot ({5-0.45*cos(\x)}, {0.45*sin(\x)});

	\tkzText(1,1.5){$\R^2$};
	\tkzText(4,1.5){$\N$};
\end{tikzpicture}
\end{center}
\caption{Geometric figures consisting of geodesic curves. In both figures the respective length of the geodesics are equal as well as the respective angles between geodesics. The figure on the left is in the Euclidean plane whereas the figure on the right is in $\N$. }
\label{fig:tedious}
\end{figure}
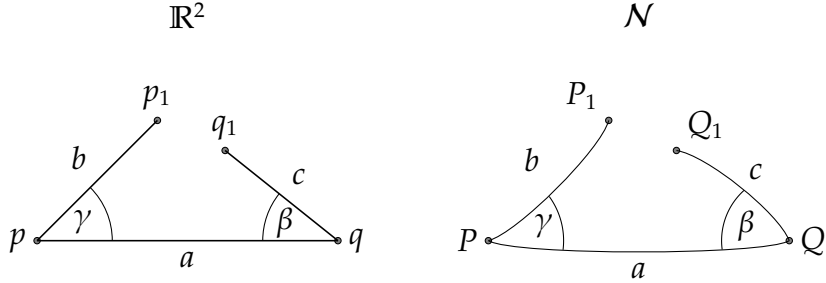

\label{lem:tedious}
\end{lemma}

This lemma can be proved using geometric comparison theorems, but an analytical approach is shorter. In fact, we will prove a stronger result:

\begin{proposition}
\label{prop:tedious} 
Let $\gamma:[0,\ell]\to\N$, $\ell \le 3\Lmax/n$, be a curve in arclength parametrization with geodesic curvature $k(t)$ (the geodesic curvature is with respect to the connection $\nabla$). Let $\sigma:[0,\ell]\to\R^2$ be a curve in arclength parametrization with geodesic curvature $k(t)$ (with respect to the Euclidean connection). Then, there exists a constant $\Delta>0$, which only depends on $\Lmax$, $\Kmax$, $\Kmin$ and $\Lambda$, such that for every $t\in[0,\ell]$,
\[
|d(\ga(t),\ga(0)) - |\sigma(t) - \sigma(0)|| \le \frac{\Delta}{n^2}.
\]
\end{proposition}

The proof is given in Appendix~\ref{sec:tedious}.

\begin{proposition}
\label{prop:Fn}
There exist smooth embeddings $F_n: \M_n\to\N$ satisfying the following properties:
\begin{enumerate}
\item $F_n$ extends $T_n:X_n\to Y_n$.
\item The image of  $\M_{n,i}$ is a subset of $\N_{n,i}$ for all $i\in I_n$.
\item $\textVol(\N\setminus F_n(\M_n))\to 0$.
\item $dF_n$ and $dF_n^{-1}$ are uniformly bounded in $n$ over their domains of definition.
\item $\dist(dF_n,\SO{\g_n,\g}) = O(1/n)$ everywhere except for a set of volume $O(1/n^{1-\ep})$, where $0<\ep<1$ can be chosen arbitrarily.
\end{enumerate}
\end{proposition}

The lower the constant $\ep$ is, the faster is the convergence of $\M_n$ to $\N$. With a more elaborate construction we can obtain the same result with $\ep=0$.
Since the convergence rate is immaterial for our main result, we will make do with $\ep>0$.

\begin{proof}
For every large enough $n\in\mathbb{N}$  and for every $i\in I_n$ we construct a  compact set $K_{n,i}\subset\tM_{n,i}$ as shown in \figref{fig:Fnconstruct}, such that 
\begin{enumerate}
\item $\textVol(K_{n,i}) = O(1/n^{3-\ep})$.
\item The intersection of $K_{n,i}$ with $X_n$ consists of three segments.
\item Distances between different connected components of $\pl K_{n,i}\setminus \pl \M_{n,i}$ are between $1/n^{2-\ep}$ and $2/n^{2-\ep}$. 
\item If $\M_{n,i}$ and $\M_{n,j}$ share an edge, then the intersections of $K_{n,i}$ and $K_{n,j}$ with this edge coincide, and the boundary of $K_{n,i}\cup K_{n,j}$ is smooth.
\end{enumerate}

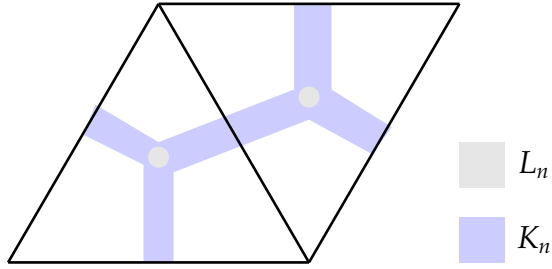
\begin{figure}
\begin{center}
\begin{tikzpicture}[scale=2]
	\tkzDefPoint(0,0){A};
	\tkzDefPoint(2,0){B};
	\tkzDefPoint(1,1.72){C};
	\tkzDefPoint(3,1.72){E};
	\tkzDefPoint(1,0.7){D};
	\tkzDefPoint(2,1.1){F};
	\fill [color=blue!20] (0.58,1.05) -- (1,0.8) -- (1.9,1.15) -- (1.9,1.72) -- (2.15,1.72) -- (2.15, 1.15) --(2.55,0.9) -- (2.42,0.72)  -- (2,0.95) -- (1.1,0.6) -- (1.1,0) -- (0.9,0) -- (0.9,0.65) -- (0.5,0.85) -- cycle;

	\tkzFillCircle[R, fill=gray!20, line width=0pt](D,0.07 cm)
	\tkzFillCircle[R, fill=gray!20, line width=0pt](F,0.07 cm)
	\tkzDrawPoints(A,B,C,E);
	\draw [line width=1pt] (0,0) edge[bend left=30,distance=0.cm] (2,0);
	\draw [line width=1pt] (0,0) edge[bend left=30,distance=0.cm] (1,1.72);
	\draw [line width=1pt] (2,0) edge[bend left=30,distance=0.cm] (1,1.72);
	\draw [line width=1pt] (1,1.72) edge[bend left=30,distance=0.cm] (3,1.72);
	\draw [line width=1pt] (2,0) edge[bend left=30,distance=0.cm] (3,1.72);

	\fill [color=blue!20] (3,0) -- (3.3,0) -- (3.3,0.3) -- (3,0.3) -- cycle;
	\fill [color=gray!20] (3,0.5) -- (3.3,0.5) -- (3.3,0.8) -- (3,0.8) -- cycle;
	\node at (3.5,0.15) {$K_n$};
	\node at (3.5,0.65) {$L_n$};
\end{tikzpicture}
\end{center}
\caption{Partition of $\tM_{n}$ used in the construction of $F_n$. The dislocation lines are surrounded by open neighborhoods $L_{n,i}$ of radius $O(1/n^2)$ whose union is  denoted by $L_n$; we denote $\M_{n,i} = \tM_{n,i}\setminus L_{n,i}$. Each $\M_{n,i}$ is partitioned into three connected components by constructing the set $K_{n}$. }
\label{fig:Fnconstruct}
\end{figure}

Set $K_n = \cup_{i\in I_n} K_{n,i}$. We define $F_n$ on $\M_n\setminus K_n$ as  follows: let $p\in \M_n\setminus K_n$, then by construction the connected component of $\M_n\setminus K_n$ that contains $p$ contains a single vertex $q_p$ of the grid $X_n$. 
Define
\[
F_n(p) = \exp^\nabla_{T_n(q_p)}(v(p)),
\]
where 
$v(p)\in T_{T_n(q_p)}\N$ is the vector of length $|p-q_p|$ which forms with $Y_n$ at $T_n(q_p)$ the same angle as $p-q_p$ forms with $X_n$ at $q_p$.  

Since the injectivity radius of $\exp^\nabla$ is at least $\Llam^{-1}(L(\Kmax,\Lambda))$  (\propref{prop:geo_exist}), it follows that $F_n$ is injective on each connected component of $\M_n\setminus K_n$. 
We have to show that $F_n$ is globally injective: it maps disjoint connected components of $\M_n\setminus K_n$ into disjoint sets in $\N$, and moreover, the separation between the images remains of order $1/n^{2-\ep}$. This is an immediate consequence of  
Lemma~\ref{lem:tedious}, which states that $F_n$ distorts paths in $\M_{n,i}$ by distances of order $O(1/n^2)$; in the application of Lemma~\ref{lem:tedious} the segment $pq$ corresponds to an edge of $\M_{n,i}$, whereas the geodesic between $P$ and $Q$ corresponds to the matching edge in $\N_{n,i}$. Note that the fact that the domain in Lemma~\ref{lem:tedious} is $\R^2$ and not $\M_{n,i}$ does not matter since the difference in distances is only $O(n^{-2})$ by construction.
Thus, $F_n$ is injective and the distance between the images of the connected components of $\M_{n,i}\setminus K_{n,i}$ is of order $1/n^{2-\ep}$. 

We turn to consider the derivative of $F_n$. Since $q=q_p$ is independent of $p$ in any connected component,
\[
dF_n = d(\exp^\nabla_{T_n(q)}) \circ dv.
\] 
Note that by definition $dv_p \in\SO{\g_n,\g}$ for every $p$. Also, 
\[
d(\exp^\nabla_{T_n(q)})_{0} = \id.
\]
Since $|v| = O(1/n)$ and $\exp^\nabla$ and its derivatives only depend on metric $\g$ and the vector field $V$ defining $\nabla$, we conclude that 
\[
\dist(d(\exp^\nabla_{T_n(q)})_v,\id) = O\brk{\frac{1}{n}},
\] 
hence
\[
\dist(dF_n,\SO{\g_n,\g}) = O\brk{\frac{1}{n}},
\]
and the same estimate holds for $dF_n^{-1}$.
An immediate consequence is $|dF_n|,|dF_n^{-1}| = 1+ O(1/n)$; in particular $F_n$ are uniformly bi-Lipschitz.

We next extend $F_n$ to $\M_n$, that is, we define $F_n$ on $K_{n,i}$.
$F_n$ maps  $\pl K_{n,i}\setminus \pl\M_{n,i}$ with an $O(1/n^2)$ distortion and a uniformly bounded derivative. Since we have shown that the dimensions of the domains and images are of the same order, it follows that $F_n$ can be extended smoothly to an embedding $\M_n$ with a uniformly bounded derivative, and with the constraint that $F_n|_{X_n} = T_n$. 

It remains to prove that $\textVol(\N\setminus F_n(\M_n))\to 0$. Note that
\[
\textVol(\N\setminus F_n(\M_n)) = \textVol(\N\setminus\N_n) + \textVol(\N_n\setminus F_n(\M_n)).
\]
The first term tends to zero by \propref{prop:triangulation}. For the second term, we note that the length of $\pl L_{n,i}$ is $O(1/n^2)$. Since $dF_n$ is uniformly bounded, $
F_n( \pl L_{n,i})$ is also a curve whose length is $O(1/n^2)$, hence, by the isoperimetric inequality for manifolds with lower-bounded gaussian curvature, its volume is $O(1/n^4)$. Hence,
\[
\textVol(\N_n\setminus F_n(\M_n)) = O(1/n^2).
\]
\end{proof}

\begin{corollary}[Mean asymptotic rigidity]
\label{corr:rigidity}
There embeddings $F_n:\M_n\to\N$ satisfy 
for every $p\in(1,\infty)$,
\[
\int_{F_n(\M_n)} \dist{^p}(dF_n^{-1},\SO{\g,\g_n}) \,\Volume \to 0.
\]
\end{corollary}

\begin{proof}
This is an immediate consequence of Items~4 and 5 in \propref{prop:Fn}.
\end{proof}

\section{Gromov-Hausdorff convergence}
\label{sec:metric}

In this section we prove that $F_n$ satisfies the vanishing distortion property in \defref{def:convergence}, i.e. that $\dis F_n\to 0$. We then deduce that $(\M_n,d_n)$ Gromov-Hausdorff (GH) converges to $(\N,d)$ as a sequence of compact metric spaces. Recall that $d$ is the distance function on $\N$ induced by the metric $\g$.

\begin{proposition}
\label{prop:disF_n}
\[
\dis F_n = \max_{p,q\in \M_n} | d_n(p,q) - d(F_n(p),F_n(q)) | = \Oof{1-\ep}.
\]
\end{proposition}
 
\begin{proof}
The proof relies on two lemmas.
The first lemma shows that the restriction of $F_n$ to a single triangle $\M_{n,i}$
has a distortion of order $O(n^{-2+\ep})$:

\begin{lemma}
\label{lm:2}
Let $F_{n,i}$ be the restriction of $F_n$ to $\M_{n,i}$. Then there exists a constant $c>0$, independent of $n$ and $i$, such that
\[
max_{p,q\in \M_{n,i}} \left|d_n(p,q) - d(F_{n,i}(p),F_{n,i}(q))\right| < \frac{c}{n^{2-\ep}}.
\]
\end{lemma}

\begin{proof}
Let $p,q\in \M_{n,i}$ and let $\gamma$ be a ``short" path in $\M_{n,i}$ between $p$ and $q$ in the sense that
\[
\ell(\gamma) \le d_n(p,q) + O\brk{\frac{1}{n^{2-\ep}}}.
\] 
The extra $O(n^{-2+\ep})$ term in the length of $\gamma$ enables us to choose $\gamma$ such that its intersection with $K_{n,i}$ is of length $O(n^{-2+\ep})$.

Let $\sigma = F_n(\gamma)$ be the image of $\ga$ in $\N_{n,i}$, which connects $F_n(p)$ and $F_n(q)$. Then,
\[
d(F_n(p),F_n(q)) \le \ell(\sigma) \le \ell(\gamma) + O\brk{\frac{1}{n^{2-\ep}}} \le 
d_n(p,q) + O\brk{\frac{1}{n^{{2-\ep}}}}.
\]
The middle inequality follows from the bounds $\ell(\gamma) = O(1/n)$ and $\ell(\gamma\cap K_{n,i}) = O(n^{-2+\ep})$, and from
the fact that $dF_n$ does not increase length by more than a factor of $O(1/n)$ in $\M_n\setminus K_{n,i}$ and at most by a factor of $O(1)$ in $K_{n,i}$.

A similar argument holds in the other direction, by choosing $\sigma$ to be a ``short" path in $\N_{n,i}$ between $F_n(p)$ and $F_n(q)$ in the sense that
\[
\ell(\sigma) \le d(F_n(p),F_n(q)) + O\brk{\frac{1}{n^{2-\ep}}}.
\] 
This time we use the extra $O(n^{-2+\ep})$ term to consider curves \revised{inside of $\N_{n,i}$ (unlike $\M_{n,i}$, $\N_{n,i}$ is not convex -- a segment between $p$ and $q$ may leave $\N_{n,i}$, but by considering only curves inside $\N_{n,i}$ the length gain is only $O(1/n^2)$),} to avoid $\N_{n,i}\setminus F_n(\M_{n,i})$ (where $F_n^{-1}$ is not defined) and guarantee that the intersection of $\sigma$ with $F_n(K_{n,i})$ is of length $O(n^{-2+\ep})$ 
\end{proof}

The second lemma bounds the number of triangles intersected by a length-minimizing curve, thus allowing to estimate the accumulated distortion along such a curve:

\begin{lemma}
\label{lm:3}
For every  $n\in\mathbb{N}$ and $p,q\in \N_n$, the shortest path in $\N$ connecting $p$ and $q$ intersects  $O(n)$ geodesic triangles $\N_{n,i}$. 
Likewise, for every (large enough) $n\in\mathbb{N}$ and $p,q\in \M_n$, the shortest path in $\M_n$ connecting $p$ and $q$ intersects  $O(n)$  of the $\M_{n,i}$'s. 
\end{lemma}

\begin{proof}
Let $p,q\in \N_n$, and let $\sigma$ be a segment connecting them.
Let $r$ be defined as in \corrref{cor:vertex_edge_dist}, that is each edge in the triangulation is of distance larger than $r/n$ from an opposite edge. 
The lower bound on the angles in \propref{prop:triangulation} implies that each vertex in $Y_n$ is surrounded by at most $2\pi/\delta$ triangles.
Therefore, a ball of radius $r/2n$ intersects at most $2\pi/\delta$ triangles. 
\revised{Therefore, a ball of radius $r/2n$ around a point $p'\in \N_n$ does not intersect more triangles than the total triangles that balls of radius $r/2n$ around the four vertices of its (two) adjacent triangles intersect, that is $8\pi/\delta$.} 
Since the length of $\sigma$ is at most the diameter of $\N$, it follows that $\sigma$ intersects at most $\revised{16}\pi n \diam(\N)/r\delta$ triangles.

Let now $p,q\in \M_n$. Let $\sigma$ be the shortest path in $\N$ connecting $F_n(p)$ and $F_n(q)$. It intersects $O(n)$ geodesic triangles. Denote by $F_n(p)=t_0,t_1,\dots,t_k=F_n(q)$ the points in $\sigma \cap Y_n$. By Lemma~\ref{lm:2}, 
\[
\begin{split}
d_n(p,q) &\le \sum_{j=1}^k d_n\brk{F_n^{-1}(t_{j-1}),F^{-1}_n(t_{j})}
\le \sum_{j=1}^k \brk{d(t_{j-1},t_j)+\frac{c}{n^{2-\ep}}} \\
&= d(F_n(p),F_n(q)) + O(1/n^{1-\ep}) < \diam(\D)+ O(1/n^{1-\ep}),
\end{split}
\]
hence the diameter of $\M_n$ is bounded uniformly in $n$, $\diam(\M_n)\le K$.

A similar reasoning now applies: since, by Lemma~\ref{lm:2}, distances in $\M_{n,i}$ and $\N_{n,i}$ differ only by $O(n^{-2+\ep})$, any vertex in $X_n$ is of distance greater than $r/n$ from an opposite edge, hence any curve of length less than $K$ intersects at most $\revised{16}\pi n K/r\delta$ triangles.
\end{proof}

Thus the distortion of $F_n$ is $O(n)\cdot O(n^{-2+\ep}) = O(n^{-1+\ep})$, which completes the proof of \propref{prop:disF_n}. See the proof of Theorem~3.1 in \cite{KM15} for a similar argument. 
\end{proof}

\begin{corollary}
\label{cor:GH}
Let $(\M_n,d_n)$ be the sequence of compact metric spaces defined in Section~\ref{sec:Mn}. Then,
$(\M_n,d_n)$ GH converges to $(\N,d)$.
\end{corollary}

\begin{proof}
The GH distance is a measure of distortions between metric spaces, and is a metric on isometry classes of compact metric spaces \cite[Chapter~10]{Pet06}. A sufficient and necessary condition for a sequence of metric spaces $(Z_n,d_n)$ to converge in the GH sense to a metric space $(Z,d)$ is that there exist bijections
\[
T_n: A_n\to B_n,
\]
where $A_n\subset Z_n$ and $B_n\subset Z$ are finite $\delta_n$-nets, $\delta_n\to0$, and the distortion of $T_n$,
\[
\dis T_n = \max_{x,y\in A_n} |d_n(x,y) - d( T_n(x), T_n(y))|
\]
tends to zero.

In our case, restrict $F_n$ to the vertices of the skeleton $X_n$. This forms a bijection between two finite $O(n^{-1})$-nets of $\M_n$ and $\N$ respectively. Since the distortion of a mapping does not grow under restriction, GH convergence is implied immediately by \propref{prop:disF_n}.
\end{proof}

Note that alternative mappings between the vertices of $X_n$ and $Y_n$ can be used to prove GH convergence. In particular, one can use an analogue of the construction shown in Appendix~A of \cite{KM15} to obtain  a $O(1/n)$ distortion, i.e., a higher rate of convergence.


\section{Convergence of \Wz\ manifolds}
\label{sec:connection}

\propref{prop:Fn}, \corrref{corr:rigidity} and \propref{prop:disF_n} cover the asymptotic surjectivity, mean asymptotic rigidity and vanishing distortion properties in \defref{def:convergence}. Therefore, to complete the proof of \thmref{th:main}, it remains to prove weak convergence of the connections, and extend the entire analysis to non simply-connected domains.

\begin{proposition}
\label{prop:connection}
Let $E$ be a fixed $\nabla$-parallel orthonormal frame field on $\N$. 
Then there exist $\nabla_n$-parallel orthonormal frame fields $E_n$ on $\M_n$ such that for every $p\in[1,\infty)$,
\beq
\int_{F_n(\M_n)} |(F_n)_\star E^n - E|^p\,\Volume = O\brk{\frac{1}{n^{1-\ep}}}.
\label{eq:connection}
\eeq
\end{proposition}

\begin{proof}
Let $E = \{e_1,e_2\}$ be a fixed $\nabla$-parallel orthonormal frame field on $\N$. Without loss of generality, assume $E$ is oriented. 
Given $n\in\mathbb{N}$ we construct an oriented $\nabla_n$-orthonormal parallel frame field $E^n = \{e^n_1,e^n_2\}$. For that, we only need to specify $e^n_1$ at a single point. 
Take a point $p\in Y_n$, and let $q\in Y_n$ be another point that lies on the same edge of a triangle. There exists a unique unit vector $v_p$ in $T_p\N$ such that $q=\exp^\nabla_{p}(t v_p)$ for some $t\in(0,\Lmax/n]$. 
Likewise, there exists a unique unit vector $w_{F_n^{-1}(p)}$ in $T_{F^{-1}_n(p)}\M_n$, such that
$F^{-1}_n(q) = F_n^{-1}(p) + t\, w_{F_n^{-1}(p)}$ for some $t\in(0,\Lmax/n]$.
$e^n_1$ is the unique unit vector that satisfies 
\[
\g_n(e^n_1,w_{F^{-1}_n(p)}) = \g(e_1,v_p).
\]

We claim that this choice implies that the above identity holds for every point $p'\in Y_n$, because parallel transport along a geodesic of a metric connection preserves both the length of the vector and its angle with the geodesic (recall that $X_n$ and $Y_n$ are unions of $\nabla_n$- and $\nabla$-geodesics, respectively).

Let now $p\in\M_n\setminus K_n$ and denote by $q_p$ the vertex of $X_n$ in its connected component. Denote by $\gamma_p$ the $\nabla_n$-geodesic connecting $q_p$ and $p$ and let $\sigma_p=F_n(\gamma_p)$ be  the $\nabla$-geodesic connecting $F_n(q_p)\in Y_n$ and $F_n(p)$. By the same argument about parallel transport along geodesics, the angle between $e^n_i$ and $\gamma_p$ at $p$ equals the angle between $e_i$ and $\sigma_p$ at $F_n(p)$. Since 
\[
\dist(dF_n,\SO{\g_n,\g}) = O\brk{\frac{1}{n}}
\qquad 
\text{ on $\M_n\setminus K_n$},
\]
and since $dF_n$ maps the unit tangent of $\gamma_p$  to the unit tangent of  $\sigma_p$,
it follows that 
\[
|(F_n)_\star E^n - E| = O\brk{\frac{1}{n}}
\qquad 
\text{ on $F_n(\M_n\setminus K_n)$}.
\]
Moreover, $dF_n$ is uniformly bounded hence so is $|(F_n)_\star E^n - E|$. Since $\textVol(K_n) = O(n^{-1+\ep})$, \eqref{eq:connection} follows.
\end{proof}

Finally, to extend the entire analysis to domains of finite genus, one has to partition the domain along geodesics into simply-connected components, approximate each domain separately up to the lines of partitions (which is why they need to be geodesics), and glue the defective components together.


\paragraph{Acknowledgements}
We are grateful to Jake Solomon, Pavel Giterman and Michael Moshe for valuable discussions and to Alon Kupferman for help with the figures.
This work was partially funded by the Israel Science Foundation and by the Israel-US Binational Foundation.

\appendix
\section{Riemannian and connection curvatures in two dimensions}
\label{sec:curlfree}

In this appendix we prove \propref{prop:curlfree}:

\begin{quote}
{\itshape
Let $(\N,\g)$ be a two-dimensional Riemannian manifold. Let $\nabla$ be a metrically-consistent connection defined as in \eqref{eq:V} by a vector field $V$. Let $K$ and $K^\g$ denote the connection and the Riemannian gaussian curvatures. Then,
\[
K \,\revised{\Volume} = K^\g \,\revised{\Volume} - d \star V^\flat.
\]
}
\end{quote}

A tedious, but straightforward calculation shows that
\[
\begin{split}
R(X,Y)Z &= R^\g(X,Y) Z \\
&+  \g(Y,Z) \nabla^\g_X V -  \g(X,Z) \nabla^\g_Y V \\
&-   \g(\nabla^\g_X V,Z) Y +   \g(\nabla^\g_Y V,Z) X \\
&+ \g(Y,Z) \g(X,V)V   - \g(X,Z) \g(Y,V)V \\
&- \g(Y,Z) |V|^2 X + \g(V,Z) \g(V,Y)  X + \g(X,Z) |V|^2 Y - \g(V,Z) \g(V,X)  Y.
\end{split}
\]
In two dimensions, for an orthonormal frame $(e_1,e_2)$,
\[
\begin{split}
\g(R(e_1,e_2)e_1,e_2) &= \g(R^\g(e_1,e_2) e_1,e_2) 
 -  \g(\nabla^\g_{e_2} V,e_2) -   \g(\nabla^\g_{e_1} V,e_1),
\end{split}
\]
that is,
\[
\begin{split}
K &= K^\g
 -  \g(\nabla^\g_{e_2} V,e_2) -   \g(\nabla^\g_{e_1} V,e_1),
\end{split}
\]

Write $V= fe_1 + ge_2$ for some functions $f,g$, and let $(\cof^1,\cof^2)$ be the co-frame of $(e_1,e_2)$. Then
\[
\begin{split}
\g(\nabla^\g_{e_1} V,e_1) +   \g(\nabla^\g_{e_2} V,e_2) 
	&= \cof^1(\nabla^\g_{e_1}V) + \cof^2(\nabla^\g_{e_2} V) \\
	&= \cof^1(\nabla^\g_{e_1}(fe_1 + ge_2)) + \cof^2(\nabla^\g_{e_2}( fe_1 + ge_2)) \\
	&= df(e_1) + g\,\cof^1(\nabla^\g_{e_1} e_2) + f\,\cof^2(\nabla^\g_{e_2}e_1) + dg(e_2) \\
	&= df(e_1) + g \,\omega^1_2(e_1) + f \,\omega^2_1(e_2) + dg(e_2),
\end{split}
\]
where $\omega^1_2= -\omega^1_2$ is the 1-form defining the connection.
On the other hand we have $V^\flat= f\cof^1 + g\cof^2$, hence $\star V^\flat = f\cof^2 - g\cof^1$,  and therefore
\[
d\star V^\flat = df \wedge \cof^2 - dg\wedge \cof^1 + f\,d\cof^2  - g\,d\cof^1.
\]
Using Cartan's first structural equation we obtain
\[
d\star V^\flat(e_1,e_2) = df(e_1) + dg(e_2) + f\,\omega^2_1(e_2) + g\,\omega^1_2(e_1) = \g(\nabla^\g_{e_1} V,e_1) +  \g(\nabla^\g_{e_2} V,e_2),
\]
i.e.,
\[
K \,\revised{\Volume} = K^\g \,\revised{\Volume} - d \star V^\flat.
\]

\section{Gauss-Bonnet theorem for metric connections}
\label{sec:GB}

\begin{theorem}
Let $\nabla$ be a metric connection on a two-dimensional oriented Riemannian manifold $(\M,\g)$. Let $\ga$ be a closed piecewise-smooth closed curve encircling a domain $\Gamma$. Denote by $\{t_i\}_{i=1}^k$ the vertices of $\ga$ and by $\e_i$ the exterior angles. Then,
\[
\sum_{i=1}^k \int_{t_{i-1}}^{t_i} k_\nabla\, dt + \sum_{i=1}^k \e_i = 2\pi -  \int_\Gamma K\,\Vol,
\]
where $t_0=t_k$, $k_\nabla$ is the signed geodesic curvature of $\ga$ with respect to $\nabla$,  $K$ is the gaussian curvature of $\nabla$, and $\Vol$ is the area form. In particular, if $\ga$ is piecewise geodesic and $\nabla$ is flat, then the sum of the exterior angles is $2\pi$.
\end{theorem}

\begin{proof}
The proof is basically an adaptation of the proof in Lee \cite[p. 164]{Lee97}.
Let $(e_1,e_2)$ be a local oriented orthonormal frame. Denote by $\theta(t)$ the unique angle between $\dot{\ga}(t)$ and $e_1(\ga(t))$, with $\theta(t_0)\in (-\pi,\pi]$, which is continuous at the smooth parts of $\gamma$ and has jumps $\e_i$ at the vertices.

By the fundamental theorem of calculus we have
\begin{align}
\label{eq:umlaufsatz}
\sum_i \e_i + \sum_i \int_{t_{i-1}}^{t_i} \dot\theta(t)\,dt = \theta(b)-\theta(a) = 2\pi
\end{align}
At all smooth points of $\gamma$ we have
\beq
\begin{aligned}
\label{eq:gammadot}
\dot\gamma(t) &= \cos \theta(t) e_1(t) + \sin \theta(t) e_2(t) \\
N(t) &= -\sin\theta(t) e_1(t) + \cos\theta(t) e_2(t).
\end{aligned}
\eeq

Differentiating \eqref{eq:gammadot} with respect to $\nabla_{\dga}$ we obtain
\[
\begin{split}
\frac{D\dot\gamma}{dt} &= 
-\sin\theta(t) \dot\theta(t)e_1 + \cos\theta(t)(-\omega(\dot\gamma(t))e_2) +
\cos\theta(t) \dot\theta(t) e_2 + \sin\theta(t)\omega(\dot\gamma(t))e_1 \\
&= \dot\theta(t)N(t) + \omega(\dot\gamma(t))(-\cos\theta(t)e_2 + \sin\theta(t)e_1),
\end{split}
\]
where $\omega$ is the one-form defined by $\nabla_X e_1 = \omega(X) e_2$ (and by the metricity of the connection $\nabla_X e_2 = -\omega(X) e_1$).

By taking an inner-product with $N(t)$, we obtain
\begin{align}
\label{eq:kappaN}
k_\nabla(t) &=  \left\langle \frac{D\dot{\gamma}(t)}{dt}, N(t) \right\rangle_{\gamma(t)}
		= \dot\theta(t) - \omega(\dot\gamma(t)).
\end{align}

Inserting \eqref{eq:kappaN} into \eqref{eq:umlaufsatz}, and using Stokes's theorem we obtain
\begin{align}
2\pi = \sum_{i=1}^k \e_i + \sum_{i=1}^k \int_{t_{i-1}}^{t_i} k_\nabla\,dt +  \sum_{i=1}^k \int_{t_{i-1}}^{t_i} \omega = \sum_{i=1}^k \e_i + \sum_{i=1}^k \int_{t_{i-1}}^{t_i} k_\nabla\,dt +  \int_\Gamma d\omega. \nonumber
\end{align}
In two-dimensions $ d\omega = K\,\Vol$, which completes the proof. 
\end{proof}

\section{Proof of \propref{prop:tedious}}
\label{sec:tedious}

In this section we prove a general result concerning short curves in two-dimensional Riemannian manifolds, possibly endowed with a non-symmetric connection: 

\begin{quote}
{\itshape
Let $\gamma:[0,\ell]\to\N$, $\ell \le 3\Lmax/n$, be a curve in arclength parametrization with geodesic curvature $k(t)$ (the geodesic curvature is with respect to the connection $\nabla$). Let $\sigma:[0,\ell]\to\R^2$ be a curve in arclength parametrization with geodesic curvature $k(t)$ (with respect to the Euclidean connection). Then, there exists a constant $\Delta>0$, which only depends on $\Lmax$, $\Kmax$, $\Kmin$ and $\Lambda$, such that for every $t\in[0,\ell]$,
\[
|d(\ga(t),\ga(0)) - |\sigma(t) - \sigma(0)|| \le \frac{\Delta}{n^2}.
\]
}
\end{quote}

\begin{proof}
We represent the curve $\ga$ using
a semi-geodesic parametrization $(r,\theta)$ as in the proof of
\propref{prop:geo_length}, where the origin is the starting point of $\ga$; thus
\[
r(t) = d(\ga(t),\ga(0)).
\]
For a metric  prescribed by a function $\vp(r,\theta)$, with a connection prescribed  by a vector field $(V^r,V^\theta)$, the parametric equations of a curve 
with geodesic curvature $k = k(t)$ are
\[
\begin{gathered}
\deriv{}{t} \dot{r} = \frac{\vp_r}{\vp} (\vp\dot{\theta})^2 -\brk{V^r \vp\dot{\theta} + \vp V^\theta \dot{r}}\vp\dot{\theta}
-  k\, \vp\dot{\theta} \\
\deriv{}{t}(\vp \dot{\theta}) =  
-\frac{\vp_r}{\vp}\dot{r}\, \vp\dot{\theta} + \brk{V^r \vp\dot{\theta} + \vp V^\theta \dot{r}} \dot{r} + k\, \dot{r}.
\end{gathered}
\]
Using the fact that $(\vp\dot{\theta})^2 = 1 - \dot{r}^2$ we rewrite it as follows:
\beq
\begin{gathered}
\deriv{}{t} (r \dot{r}) +  (k + \bar{\Delta}) \, r \vp\dot{\theta} =  1  \\
\deriv{}{t}(r \vp \dot{\theta}) - (k +\bar{\Delta})\, r \dot{r} =  0,
\end{gathered}
\label{eq:sys_tedious}
\eeq
where 
\[
\bar{\Delta} = \brk{V^r \vp\dot{\theta} + \vp V^\theta \dot{r}} - \brk{\frac{\vp_r}{\vp} - \frac{1}{r}} \vp\dot{\theta} \equiv \Delta_1 + \Delta_2.
\]
We have the following estimates:
\[
\begin{aligned}
|\Delta_1| &\le \Lambda \\
|\Delta_2| &\le \sup \Abs{\frac{\vp_r r - \vp}{r\vp}} = \sup  \Abs{\frac{\vp_{rr} r}{\vp + r\vp_r}}
\le \frac{3\Lmax}{n} \max(|\Kmin|,|\Kmax|).
\end{aligned}
\]

Defining $z = r(\dot{r} + \imath\,  \vp\dot{\theta})$, \eqref{eq:sys_tedious} takes the complex form,
\[
\dot{z} - \imath (k+\bar{\Delta}) z =1.
\]
Integrating, we obtain
\[
\begin{aligned}
(r\dot{r})(t) &= 
 \int_0^t \cos(\xi(t) - \xi(s))\, ds,
\end{aligned}
\]
where 
\[
\xi(t) =\xi_0(t)  + \xi_1(t),
\] 
with
\[
\xi_0(t) = \int_0^t k(s)\,ds 
\Textand
\xi_1(t) = \int_0^t \bar{\Delta}(s)\,ds = O(n^{-1}).
\]
Integrating a second time
\[
r^2(t) = 2 \int_0^t \int_0^s \cos(\xi(s) - \xi(s'))\, ds' ds = \int_0^t \int_0^t \cos(\xi(s) - \xi(s'))\, ds' ds,
\]
where the second equality follows by symmetry. Using the formula for the cosine of a difference, we get 
\[
r^2(t) =  \brk{\int_0^t\cos \xi(s)\,ds}^2 + \brk{\int_0^t\sin \xi(s)\,ds}^2.
\]
The Euclidean case is retrieved by setting $\xi_1=0$, namely,
\[
r^2_E(t) = |\sigma(t) - \sigma(0)| = \brk{\int_0^t\cos \xi_0(s)\,ds}^2 + \brk{\int_0^t\sin \xi_0(s)\,ds}^2.
\] 

Substituting $\xi = \xi_0 + \xi_1$, it takes simple manipulations to get
\[
\begin{split}
r^2(t) &= r_E^2(t) + 4\brk{\int_0^t \sin \frac{\xi_1}{2} \sin\brk{\frac{\xi_1}{2} - \xi_0}\, ds}^2 
+ 4\brk{\int_0^t \sin \frac{\xi_1}{2} \cos\brk{\frac{\xi_1}{2} - \xi_0}\, ds}^2 \\
&- 2\brk{\int_0^t\cos \xi_0(s)\,ds} \brk{\int_0^t \sin \frac{\xi_1}{2} \sin\brk{\frac{\xi_1}{2} - \xi_0}\, ds}
\\
&+ 2\brk{\int_0^t\sin \xi_0(s)\,ds} \brk{\int_0^t \sin \frac{\xi_1}{2} \cos\brk{\frac{\xi_1}{2} - \xi_0}\, ds}.
\end{split}
\]
Since
\[
\Abs{\int_0^t \cos \xi_0\, ds}, \Abs{\int_0^t \sin \xi_0\, ds} \le r_E(t),
\]
and both $\xi_1$ and $t$ are $O(1/n)$, it follows that
\[
\Abs{r^2(t) - r_E^2(t)} \le \frac{\Delta}{n^2}r_E + O\brk{\frac{1}{n^4}},
\]
which implies
\[
\Abs{r(t) - r_E(t)} \le \frac{\Delta}{n^2}.
\]
\end{proof}


\bibliographystyle{amsalpha}

\providecommand{\bysame}{\leavevmode\hbox to3em{\hrulefill}\thinspace}
\providecommand{\MR}{\relax\ifhmode\unskip\space\fi MR }
\providecommand{\MRhref}[2]{%
  \href{http://www.ams.org/mathscinet-getitem?mr=#1}{#2}
}
\providecommand{\href}[2]{#2}

\end{document}